%% file: main.tex
\documentclass{article}
\usepackage[utf8]{inputenc}

\usepackage{amsmath, amsthm, mathtools, amsfonts, amssymb, bm}
\usepackage{hyperref}
\usepackage{fullpage}
\usepackage{verbatim}

\title{Symmetric Operations on Domains of Size at Most 4}
\author{Zarathustra Brady and Holden Mui}
\date{\today}

\begin{document}

\input{commands.tex}

\maketitle


\begin{abstract}
To convert a fractional solution to an instance of a constraint satisfaction problem into a solution, a rounding scheme is needed, which can be described by a collection of symmetric operations with one of each arity.
An intriguing possibility, raised in a recent paper by Carvalho and Krokhin, would imply that any clone of operations on a set $D$ which contains symmetric operations of arities $1, 2, \ldots, \size{D}$ contains symmetric operations of all arities in the clone. If true, then it is possible to check whether any given family of constraint satisfaction problems is solved by its linear programming relaxation.

We characterize all idempotent clones containing symmetric operations of arities $1, 2, \ldots, \size{D}$ for all sets $D$ with size at most four and prove that each one contains symmetric operations of every arity, proving the conjecture above for $\size{D} \leq 4$.
\end{abstract}


\section{Introduction}
\input{introduction.tex}

\section{Definitions}
\input{definitions.tex}

\section{Related Results}
\input{goals.tex}

\section{Preliminary Theorems}
\input{results.tex}

\section{Classification}
\input{classification.tex}

\section{Future Work}
\input{futurework.tex}

\section{Acknowledgements}
\input{acknowledgements.tex}

\bibliographystyle{plain}
\bibliography{references}

\appendix

\section{Reversible instances and linear programming}
\input{appendix.tex}

\end{document}

%% file: commands.tex
\newtheorem{thm}{Theorem}
\newtheorem{theorem}[thm]{Theorem}

\newtheorem{conjecture}[thm]{Conjecture}
\newtheorem{corollary}[thm]{Corollary}
\newtheorem{lemma}[thm]{Lemma}
\newtheorem{proposition}[thm]{Proposition}
\theoremstyle{definition}
\newtheorem{remark}[thm]{Remark}
\newtheorem{definition}[thm]{Definition}
\newtheorem{example}[thm]{Example}

\def\changemargin#1#2{\list{}{\rightmargin#2\leftmargin#1}\item[]}
\let\endchangemargin=\endlist 

\newenvironment{claim}[1]{\begin{changemargin}{0.5cm}{0.5cm} 
\par\noindent\textbf{Claim:}\space#1}{\end{changemargin}}
\newenvironment{claimproof}[1]{\begin{changemargin}{0.5cm}{0.5cm} \par\noindent\textbf{Proof:}\space#1}{\hfill $\blacksquare$\end{changemargin}}

\newcommand{\size}[1]{\vert #1 \vert}
\renewcommand{\th}{\text{th}}
\newcommand{\Sg}[2]{\textnormal{Sg}_#1\hspace{-2pt}\left(#2\right)}
\newcommand{\suchthat}{\mid}
\newcommand{\ar}{\textnormal{ar}}
\newcommand{\satisfies}{\models}
\newcommand{\ZZ}{\mathbb Z}
\newcommand{\RR}{\mathbb R}
\renewcommand{\aa}{\bm{a}}
\newcommand{\rr}{\bm{r}}
\newcommand{\xx}{\bm{x}}
\newcommand{\yy}{\bm{y}}
\newcommand{\lpr}{\triangle}
\newcommand{\fO}{\mathcal O}
\newcommand{\preserves}{\triangleright}
\newcommand{\gen}[1]{\left \langle #1 \right \rangle}
\newcommand{\sgn}{\operatorname{sgn}}
\newcommand{\conv}{\operatorname{conv}}
\newcommand{\Sym}{\operatorname{Sym}}
\newcommand{\thereexists}{\exists\:}
\newcommand{\Clo}{\operatorname{Clo}}

\renewcommand{\AA}{\mathbb{A}}
\newcommand{\BB}{\mathbb{B}}
\newcommand{\PP}{\mathbb{P}}
\renewcommand{\SS}{\mathbb{S}}

\newcommand{\binarythree}[3]{
\begin{tabular}{c|ccc}
$f_{#1#2#3}$ & $-$  & $0$  & $+$ \\ \hline
$-$          & $-$  & $#3$ & $#2$ \\
$0$          & $#3$ & $0$  & $#1$ \\
$+$          & $#2$ & $#1$ & $+$
\end{tabular}
}

\newcommand{\binaryfour}[6]{
\begin{tabular}{c|cccc}
$f_{#1#2#3#4#5#6}$ & $0$  & $1$  & $2$  & $3$ \\ \hline
$0$                & $0$  & $#1$ & $#2$ & $#3$ \\
$1$                & $#1$ & $1$  & $#4$ & $#5$ \\
$2$                & $#2$ & $#4$ & $2$  & $#6$ \\
$3$                & $#3$ & $#5$ & $#6$ & $3$
\end{tabular}
}

%% file: introduction.tex
The Constraint Satisfaction Problem, commonly abbreviated as the CSP, is the decision problem where we are given a list of variables and a list of constraints on the variables, and we must determine whether or not there exists an assignment of the variables which satisfies every constraint. While this decision problem is NP-hard in general, certain classes of CSPs can be solved in polynomial time. Among those are the CSPs that are solved by their linear programming relaxation; that is, a ``fractional solution'' to an instance of such a CSP can be rounded to a solution. Such CSPs have been characterized as the CSPs for which the clone of operations preserving its relations contains symmetric operations of every arity. In this paper, we investigate a conjecture that, if true, gives a sufficient condition for a clone to contain symmetric operations of every arity.
\begin{conjecture}\label{conj}
Suppose a clone over a domain $D$ contains symmetric operations of arities $1, 2, \ldots, \size{D}$. Then it contains symmetric operations of every arity.
\end{conjecture}

This conjecture is a weak form of an open problem mentioned in section 6 of \cite{carvalho}: its authors speculate that a finite algebraic structure $\AA$ has symmetric terms of every arity if and only if $\AA$ has no subquotient $\BB$ such that the automorphism group of $\BB$ contains a pair of automorphisms with no common fixed point.

In this paper, we prove Conjecture \ref{conj} for $\size{D} \leq 4$.
\begin{theorem}\label{result}
Suppose a clone over a domain $D$ with $\size{D}\leq4$ contains symmetric operations of arities $1, 2, \ldots, \size{D}$. Then it contains symmetric operations of every arity.
\end{theorem}

In the appendix, we sketch a more ambitious conjecture about the solvability of certain types of weakly consistent constraint satisfaction problems attached to an algebraic structure satisfying the assumption of Conjecture \ref{conj}.

\subsection{Motivation}

If Conjecture \ref{conj} is true, then it gives an efficient way of determining whether or not combinatorial puzzles can be solved using systems of linear inequalities.

To understand Conjecture \ref{conj}'s implications, it is helpful to characterize the types of combinatorial puzzles we want to look at; a good example of such a puzzle is Sudoku. Sudoku is a constraint satisfaction problem over the domain $D = \{1, 2, 3, 4, 5, 6, 7, 8, 9\}$ with ten types of constraints. There is the 9-ary constraint which asserts that its inputs are pairwise distinct, and there are nine unary constraints, each of the form ``this variable is equal to $d$'' for some $d \in D$. Instances of this constraint satisfaction problem have 81 variables, and the variables that the constraints apply to depend on the specific instance. Other examples of CSPs include HornSAT and 3Coloring: HornSAT is the constraint satisfaction problem where the task is to determine whether or not a set of Horn clauses (implications) admits an assignment of the variables satisfying each Horn clause, and 3Coloring is the constraint satisfaction problem whose task is to determine whether or not a given graph admits a 3-coloring.

The linear programming relaxation of an instance of a CSP is, informally, the set of all locally consistent probability distributions over its variables and constraint relations, and is defined by a collection of linear inequalities (a rigorous but terse definition is given in section 2). A CSP is solved by its linear programming relaxation iff there is a way to turn points in the linear programming relaxation of every instance of the CSP to a solution of that CSP, known as a rounding scheme.

While finding solutions to instances of a general CSP is NP-hard, solving the linear programming relaxation of any CSP only takes polynomial time, so there is an efficient way to find a solution to an instance of a CSP if it is solved by its linear programming relaxation. Usually this is not the case, but for these special CSPs, local probability distributions of solutions, called fractional solutions, can be converted to true solutions using the rounding scheme.

The CSPs solved by their linear programming relaxation have been characterized in \cite{kun}. They are precisely the ones for which the clone of operations preserving each relation defining the CSP contains symmetric operations of every arity.

\begin{theorem}[Theorem 2 of \cite{kun}]
The CSP defined by a collection of relations $\Gamma$ is solved by the linear programming relaxation if and only if the clone of operations that preserves each relation in $\Gamma$ contains symmetric operations of every arity.
\end{theorem}

An equivalent characterization of clones with symmetric operations of every arity appears in a recent article by Butti and Dalmau about solving CSPs with \emph{distributed} algorithms \cite{dalmau-distributed}. In their setup, each agent has access to a single variable or to a single constraint, agents can communicate only when one owns a constraint involving the variable owned by the other, and the agents are \emph{anonymous}, so there is no obvious way to elect a leader.

\begin{theorem}[Theorem 6 of \cite{dalmau-distributed}] The CSP defined by a collection of relations $\Gamma$ can be solved in the distributed setting described above if and only if the clone of operations that preserves each relation in $\Gamma$ contains symmetric operations of every arity. If it can be solved in the distributed setting at all, then it can be solved in this setting in polynomial time.
\end{theorem}

For brevity, we call such clones \emph{round}, as they can be used to construct a rounding scheme that turns fractional solutions into solutions of instances of CSPs. However, determining whether or not a given clone is round is difficult (possibly undecidable); that is, unless Conjecture \ref{conj} is true.

Recall that Conjecture \ref{conj} states that the existence of symmetric operations of arities $1, 2, \ldots, \size{D}$ is a sufficient (and necessary, but this direction is obvious) condition for a clone to contain symmetric operations of every arity. Equivalently, Conjecture \ref{conj} asserts that given $\size{D}$ symmetric operations of arities $1, 2, \ldots, \size{D}$, one can create a symmetric operation of any desired arity by composing the $\size{D}$ ``base'' operations in some way. If Conjecture \ref{conj} is true, then determining whether or not a clone is round becomes a finite case check!

Although Conjecture \ref{conj} remains an open problem in full generality, we prove Conjecture \ref{conj} for all clones over a domain of size at most 4. We achieve this by classifying all minimal idempotent clones over a domain of size at most 4 satisfying Conjecture \ref{conj}'s hypothesis and proving that each one is round. It is our hope that this classification will help future researchers form stronger hypotheses, verify their truth for a large number of examples, and ultimately take steps closer towards a proof (or a disproof) of Conjecture \ref{conj}.

\subsection{Road Map}

The remainder of our paper is organized as follows. In section 2, we go over definitions. In section 3, we summarize related results. In section 4, we prove results that simplify the enumeration of all minimal idempotent round clones over a domain of size at most four. In section 5, we enumerate all such clones. In section 6, we sketch a plausible line of attack on the general case of Conjecture \ref{conj}. In the appendix, we describe connections between the linear programming relaxation of a CSP and certain weak consistency conditions, and we conjecture a precise connection between them.




%% file: definitions.tex
\subsection{Constraint Satisfaction Problems}
A \textit{domain}, denoted by the capital letter $D$, is a set of values a variable can be assigned to.
A $k$-\textit{ary} \textit{relation} $R$ over a domain $D$ is a subset of the $k$-fold Cartesian product $D^k \coloneqq D \times \ldots \times D$; $k$ is known as the constraint's \textit{arity}, denoted $\ar(R)$.
A tuple $(a_1, \ldots, a_k)$ \textit{satisfies} a relation $R$ if $(a_1, \ldots, a_k) \in R.$ A \emph{constraint} is a pair consisting of a relation $R$ and an $\ar(R)$-tuple of variables.
A \textit{constraint satisfaction problem} is a pair $P = (D, \Gamma)$ where $D$ is its domain and $\Gamma$ is a set of relations over $D$.
An \textit{instance} of a constraint satisfaction problem $(D, \Gamma)$ is a pair $I = (X, T)$ where 
\begin{itemize}
    \item $X = \left\{x_1, \ldots, x_{\size{X}}\right\}$ is a finite set of variables, and
    \item $T$ is a set of constraints involving the variables in $X$, such that each constraint relation is an element of $\Gamma$. Formally, $T$ is a set of pairs $(\xx, R)$, where $R \in \Gamma$ and $\xx \in X^{\ar(R)}$ is the tuple of variables the relation $R$ is applied to.
\end{itemize}
An assignment $x_1 = a_1$, $\ldots$, $x_n=a_n$ of the variables to elements of $D$ is a \textit{solution} to that instance if, for each pair $(\xx, R) \in T$, $R$ is satisfied by the tuple $\xx$ after replacing each $x_i$ in $\xx$ with $a_i$.
The \textit{linear programming relaxation} of a $k$-ary relation $R = \left\{\rr_1, \ldots, \rr_{\size{R}}\right\}$ over a finite domain $D = \left\{d_1, \ldots, d_{\size{D}}\right\}$ is the polyhedron in $\left(\RR^{\size{D}}\right)^k$ defined by the set of all points 
\begin{align*}
    \bigl(&\left((v_1)_{d_1}, (v_1)_{d_2}, \ldots, (v_1)_{d_{\size{D}}}\right), \\
    &\left((v_2)_{d_1}, (v_2)_{d_2}, \ldots, (v_2)_{d_{\size{D}}}\right), \\
    &\hspace{70pt}\vdots \\
    &\left((v_k)_{d_1}, (v_k)_{d_2}, \ldots, (v_k)_{d_{\size{D}}}\right)\:\bigr)
\end{align*}
for which there exist reals $p_{\rr_1}, \ldots, p_{\rr_{\size{R}}}$ such that
\begin{gather*}
    0 \leq p_{\rr_1}, \ldots, p_{\rr_{\size{R}}} \leq 1, \\
    \sum_{\rr \in R} p_{\rr} = 1,
\end{gather*}
and
\[(v_i)_{d_j} = \sum_{\rr \in R \suchthat \rr_i = d_j} p_{\rr}\] for all $1 \leq i \leq k$ and $1 \leq d \leq \size{D}$.
The \textit{linear programming relaxation} of an instance $I=(X, T)$ of a CSP over a domain $D$ is the polyhedron in $\left(\RR^{\size{D}}\right)^{|X|}$ defined by the set of all points
\begin{align*}
    \biggl(&\left((x_1)_{d_1}, \hspace{12pt} (x_1)_{d_2}, \hspace{12pt} \ldots, (x_1)_{d_{\size{D}}}\right), \\
    &\left((x_2)_{d_1}, \hspace{12pt} (x_2)_{d_2}, \hspace{12pt} \ldots, (x_2)_{d_{\size{D}}}\right), \\
    &\hspace{95pt}\vdots \\
    &\left(\left(x_{\size{X}}\right)_{d_1}, \left(x_{\size{X}}\right)_{d_2}, \ldots, \left(x_{\size{X}}\right)_{d_{\size{D}}}\right)\biggr)
\end{align*}
such that for each pair $(\xx, R) \in T$, $\xx$ lies in $R$'s linear programming relaxation when each variable in $\xx$ is replaced by its corresponding $\size{D}$-tuple.
A \textit{fractional solution} of an instance of a CSP is a point inside its linear programming relaxation.
We say that a CSP is \textit{solved} by its linear programming relaxation if, for every instance $I$ of the CSP, the existence of a fractional solution implies the existence of a solution.

Let $P = (D, \Gamma)$ be a constraint satisfaction problem, let $X=\left\{x_1, \ldots, x_{\size{X}}\right\}$ be a set of variables, and let \[I=\left(X, \{(\xx_1, R_1), \ldots, (\xx_{\size{R}}, R_{\size{R}})\}\right)\] be an instance of $P$. A \textit{step} $s=(k, (i, j))$ in $I$ is defined to be a constraint relation $(\xx_k, R_k)$ and a pair of integers $1 \leq i, j \leq \ar(R_k)$; we think of the step $s$ as connecting the variable $(\xx_k)_i$ to the variable $(\xx_k)_j$. A \textit{cycle} $p$ in $I$ is a finite sequence of steps
\[s_1, \ldots, s_{\size{p}} = \left(k_1, (i_1, j_1)\right), \ldots, \left(k_{\size{p}}, (i_{\size{p}}, j_{\size{p}})\right)\]
for which $(\xx_{k_1})_{j_1}=(\xx_{k_2})_{i_2}$, $(\xx_{k_2})_{j_2}=(\xx_{k_3})_{i_3}$, $(\xx_{k_{\size{p}-1}})_{j_{\size{p}-1}}=(\xx_{k_{\size{p}}})_{i_{\size{p}}}$, and $(\xx_{k_{\size{p}}})_{j_{\size{p}}}=(\xx_{k_1})_{i_1}$, where the subscript $n$ on each $(\xx_i)_n$ represents the variable in $X$ corresponding to the $n^\th$ coordinate of $\xx_i$. If $B$ is a subset of $D$ and $s=(k, (i, j))$ is a step in $I$, then we define the sum $B+s$ as
\[B+s \coloneqq \{d \in D \mid \exists \aa \in R_k\text{ s.t. } \aa_i \in B \land \aa_j=d\}\]
and the sum $B-s$ as
\[B+s \coloneqq \{d \in D \mid \exists \aa \in R_k\text{ s.t. } \aa_j \in B \land \aa_i=d\},\] where the subscript $\aa_i$ denotes the $i^\th$ coordinate of the tuple $\aa$.
If $p = s_1, \ldots, s_{\size{p}}$ is a cycle in $I$, then we define the sum $B+p$ as
\[B+p \coloneqq B+s_1 + \ldots + s_p\]
and the sum $B-p$ as
\[B-p \coloneqq B-s_p - \ldots - s_1.\]
\subsection{Clones}
Let $D$ be some domain.
An \textit{operation} $f$ is a function $f : D^k \to D$ for some positive integer $k$, known as its \textit{arity}.
An operation with arity 1 is \textit{unary}, an operation with arity 2 is \textit{binary}, and an operation with arity 3 is \textit{ternary}.
In general, an operation $f$ with arity $k$ is $k$-\textit{ary}, and its arity is denoted $\ar(f)$.
The output of an operation $f$ with inputs $x_1, x_2, \ldots, x_k$ is denoted \[f(x_1, x_2, \ldots, x_k).\]
Let $f : D^k \to D$ be an operation. We extend $f$ to an operation on vectors in $D^n$ by applying it coordinatewise, i.e.,
\[
    f\left(\begin{bmatrix}(a_1)_1\\(a_1)_2\\\vdots\\(a_1)_n\end{bmatrix},
    \begin{bmatrix}(a_2)_1\\(a_2)_2\\\vdots\\(a_2)_n\end{bmatrix},
    \ldots,
    \begin{bmatrix}(a_k)_1\\(a_k)_2\\\vdots\\(a_k)_n\end{bmatrix}
    \right) \coloneqq
    \begin{bmatrix}
    f((a_1)_1, (a_2)_1, \ldots, (a_k)_1) \\
    f((a_1)_2, (a_2)_2, \ldots, (a_k)_2) \\
    \vdots \\
    f((a_1)_n, (a_2)_n, \ldots, (a_k)_n)
    \end{bmatrix}.
\]
A $k$-ary operation $f: D^k \to D$ \textit{preserves} a relation $R$ if \[f(r_1, \ldots, r_k) \in R\] for all $r_1, \ldots, r_k \in R$ (note that the arity of $f$ has nothing to do with the arity of $R$). The relation \textit{generated} by $f$ with generators $\xx_1, \xx_2, \ldots$ is the smallest relation containing $\xx_1, \xx_2, \ldots$ that $f$ preserves, and is denoted \[\Sg{f}{\xx_1, \xx_2, \ldots}.\]

For every $k \in \ZZ^+$ and integer $1 \leq i \leq k$, the \textit{projection operation} $\pi_i^k : D^k \to D$ over a domain $D$ is defined as \[\pi_i^k(d_1, \ldots, d_k) \coloneqq d_i.\]
Note that $\pi_1^1$ is the identity operation over $D$. A \textit{clone} over a domain $D$ is a set $\fO$ of finite-arity operations that contains every projection operation over $D$ and is closed under multiple composition; that is, if $f \in \fO$ is an $m$-ary operation and $g_1, \ldots, g_m \in \fO$ are $n$-ary operations, then the operation \[h(x_1, \ldots, x_n) \coloneqq f(g_1(x_1, \ldots, x_n), \ldots, g_m(x_1, \ldots, x_n))\] is also in $\fO$.
Note that $\fO$ is closed under any ``natural'' way of composing operations because every projection operation is in $\fO$.
A clone is \textit{compatible} with a cyclic automorphism if there is a renaming defined by a cyclic permutation of the domain elements such that the renamed clone contains the same operations as the original clone.

A \textit{subclone} $\fO'$ of a clone $\fO$ is a subset of $\fO$ that is a clone.
The subclone is \textit{proper} if $\fO' \neq \fO$. The clone \textit{generated} by a set of operations $\{f_1, f_2, \ldots\}$ is the smallest clone containing $\{f_1, f_2, \ldots\}$ and is denoted
\[\gen{f_1, f_2, \ldots}.\]

An operation $f$ over a domain $D$ is \textit{idempotent} if \[f(\underbrace{x, \ldots, x}_{\ar(f)\: x\text{'s}}) = x\] for all $x \in D$. A clone $\fO$ is \textit{idempotent} if every operation in $\fO$ is idempotent.
A $k$-ary operation $f$ is \textit{symmetric} if \[f(x_1, \ldots, x_k) = f\left(x_{\sigma(1)}, \ldots, x_{\sigma(k)}\right)\] for all $x_1, \ldots, x_k \in D$ and permutations $\sigma : \{1, \ldots, k\} \to \{1, \ldots, k\}$.
We call a clone \textit{round} if it contains symmetric operations of every arity. We call a clone over a domain $D$ \textit{semi-round} if it contains symmetric operations of arities $1, 2, \ldots, \size{D}$. The relation \textit{generated} by a clone $\fO$ with generators $\xx_1, \xx_2, \ldots$ is the smallest relation that every operation in $\fO$ preserves, and it is denoted \[\Sg{\fO}{\xx_1, \xx_2, \ldots}.\]

Let $P$ be a property of a clone.
We say a clone is \textit{minimal} with respect to property $P$ if it does not contain a proper subclone with property $P$.
For example, a clone is \textit{minimally round} if it does not contain a proper round subclone, and a clone is \textit{minimally semi-round} if it does not contain a proper semi-round subclone. Lastly, a \textit{minimal idempotent round clone} is an idempotent clone that is minimally round, and a \textit{minimal idempotent semi-round clone} is an idempotent clone that is minimally semi-round.


\subsection{Algebraic concepts}

An \emph{algebraic structure} $\AA = (A, f_1, f_2, \ldots)$, also known as an \emph{algebra}, is a domain $A$, which we call the \textit{underlying set} of $\AA$, with some operations $f_i : A^{k_i} \rightarrow A$, which we call the \textit{basic operations} of $\AA$. Algebraic structures will always be written in blackboard bold. The sequence of arities $k_1, k_2, \ldots$ is called the \emph{signature} of the algebra $\AA$. 

Given an algebraic structure $\AA = (A, f_1, f_2, \ldots)$, we define the \emph{power} $\AA^m = (A^m, f_1, f_2, \ldots)$ to be an algebraic structure of the same signature, where each $f_i$ acts coordinatewise on $A^m$. A \emph{subalgebra} $\BB$ of an algebra $\AA$, denoted as $\BB \le \AA$, consists of a subset $B \subseteq A$ closed under the basic operations of $\AA$, and basic operations which are restrictions of the basic operations of $\AA$ to $B$. If $B$ is any subset of $A^m$, then we define the \emph{subalgebra generated by} $B$, denoted $\operatorname{Sg}_{\AA^m}(B)$, to be the smallest subalgebra of $\AA^m$ which contains $B$.

The \emph{clone} of an algebraic structure $\AA$, written $\Clo(\AA)$, is the clone generated by the basic operations of $\AA$. When we study algebraic structures, we will mainly be interested in properties which only depend on their clones instead of their basic operations.

A \emph{relation} $\RR \le \AA^m$ on an algebra $\AA$ will always refer to a subalgebra of $\AA^m$ for some integer $m$, known as the \emph{arity} of $\RR$. Alternatively, a relation $\RR \le \AA^m$ is a subset of $A^m$ which is \emph{compatible} with the basic operations of $\AA$; that is, for each basic operation $f_i$ of arity $k_i$ and for every choice of $k_i$ tuples $\aa^1, \ldots, \aa^{k_i} \in \RR$, we have $f(\aa^1, ..., \aa^{k_i}) \in \RR$, where $f$ acts coordinatewise. Whether or not a given set $R \subseteq A^m$ defines a relation compatible with the basic operations of $\AA$ only depends on the clone of $\AA$.

A \emph{congruence} $\theta$ on an algebraic structure $\AA$ is an equivalence relation $\theta \le \AA^2$ compatible with the basic operations of $\AA$, and the \emph{quotient} $\AA/\theta$ is an algebraic structure with domain $A/\theta$ with the basic operations defined in the natural way.

If $\RR \le \AA^m$ is a relation and $I$ is a subset of the coordinates $\{1, \ldots, m\}$, then we define the \emph{existential projection} $\pi_I(\RR)$ as
\[
\pi_I(\RR) \coloneqq \left\{\bm{x} \in A^{I} \mid \exists \bm{y} \in \RR\text{ s.t. }\bm{y}_i = \bm{x}_i \;\; \forall i \in I\right\}.
\]
For brevity, define $\pi_{i, j, \ldots}(\RR)$ to be $\pi_{\{i, j, \ldots\}}(\RR)$. A relation $\RR \le \AA^m$ is \emph{subdirect}, denoted $\RR \le_{sd} \AA^m$, if the $i^\th$ projection $\pi_i(\RR)$ is equal to $A$ for every integer $1 \leq i \le m$.

If $\RR, \SS \le \AA^2$ are binary relations, then we define their \emph{composition} as
\[
\RR \circ \SS \coloneqq \left\{(x,z) \in A^2 \mid \exists y \in A\text{ s.t. }(x,y) \in \RR \wedge (y,z) \in \SS\right\}.
\]
We define the \emph{reverse} of the binary relation $\RR$, denoted $\RR^-$, as
\[
\RR^- \coloneqq \{(y,x) \in A^2 \mid (x,y) \in \RR\}.
\]
If $\RR \le_{sd} \AA^2$, then we define the \emph{linking congruence} of $\RR$ on the first coordinate to be the congruence
\[
\bigcup_{n \ge 1} (\RR \circ \RR^-)^{\circ n}.
\]
The linking congruence on the second coordinate is defined similarly, with $\RR$ and $\RR^-$ swapped.
If $B \subseteq A$ is a set and $\RR \le \AA^2$ is a binary relation, then we define the sum $B + \RR$ as
\[
B + \RR \coloneqq \{y \in A \mid \exists x \in B\text{ s.t. } (x,y) \in \RR\}.
\]
and we define the difference $B - \RR$ as \[B - \RR \coloneqq B + \RR^- = \{x \in A \mid \exists y \in B\text{ s.t. }(x,y) \in \RR\}.\]

If $\mathcal{A}$ is a collection of algebraic structures which all have the same signature, then we define $P(\mathcal{A})$ to be the collection of all products of algebras in $\mathcal{A}$, $S(\mathcal{A})$ to be the collection of all subalgebras of algebras in $\mathcal{A}$, and $H(\mathcal{A})$ to be the collection of all homomorphic images of algebras in $\mathcal{A}$ (which is the collection of all algebras which are isomorphic to the quotient of some algebra $\AA$ in $\mathcal{A}$ by some congruence on $A$). If $\BB \in HS(\AA)$, then we call $\BB$ a \emph{subquotient} of $\AA$.

A clone is \textit{Taylor} if it contains idempotent operations that satisfy some functional equation that cannot be satisfied by projection operations. An algebra is called \emph{Taylor} if its clone is Taylor. By Birkhoff's $HSP$ theorem, an idempotent algebra is Taylor if and only if $HSP(\AA)$ does not contain a two-element algebra with each of its basic operations equal to a projection.


\subsection{Miscellaneous}
For sets $A$ and $B$, the set $A+B$ is defined as
\[A+B \coloneqq \{a + b \suchthat a \in A, b \in B\}.\]
The function $\sgn : \RR \to \RR$ is defined as
\[\sgn(x) \coloneqq \begin{cases} -1 & \text{if }x<0 \\ 0 & \text{if }x=0 \\ 1 & \text{if }x>0.\end{cases}\]
For a $(k-1)$-ary operation $f$, we define $c_f : D^k \to D^k$ as 
\begin{align*}
    c_f((x_1, \ldots, x_k)) \coloneqq (&f(\hspace{14pt} x_2, x_3, \ldots, x_{k-1}, x_k), \\
    &f(x_1, \hspace{14pt} x_3, \ldots, x_{k-1}, x_k), \\
    &f(x_1, x_2, \hspace{14pt} \ldots, x_{k-1}, x_k), \\
    &\hspace{56pt} \vdots, \\
    &f(x_1, x_2, x_3, \ldots, x_{k-1} \hspace{14pt})).
\end{align*}
We also define
\[c_f(x_1, \ldots, x_k) \coloneqq c_f((x_1, \ldots, x_k))\]
for convenience. Additionally, given a tuple $\xx = (x_1, \ldots, x_k)$ and a $k$-ary operation $f$, define \[f(\xx) \coloneqq f(x_1, \ldots, x_k).\]

We say an operation $f$ over a domain $D$ \textit{acts like a height-1 semilattice} over a subset $D' \subseteq D$ of its domain if it is idempotent and
\[f\left(x_1, \ldots, x_{\ar(f)}\right)\]
is the same fixed value $c \in D'$ over all non-constant tuples $\left(x_1, \ldots, x_{\ar(f)}\right) \in (D')^{\ar(f)}$. We say a clone \textit{acts like a height-1 semilattice} over a subset $D' \subseteq D$ of its domain if all its operations act like a height-1 semilattice over $D'$, and the constant $c$ is the same across all operations.

We say a binary operation $f$ over a domain $D$ \textit{acts linearly} over a subset $D' \subseteq D$ if $\size{D'}$ is odd and the domain elements of $f$, when restricted to the domain $D'$, can be renamed such that the new operation $f': (\ZZ/\size{D'}\ZZ)^2 \to \ZZ/\size{D'}\ZZ$ satisfies
\[f'(x, y) = \frac{x+y}{2} \pmod{\size{D'}}\]
for all $x, y \in \ZZ / \size{D'}\ZZ$.

A binary operation $f$ over a domain $D$ is a \textit{semilattice operation} if there exists a poset on $D$ such that $f$ represents ``join''; that is,
\[f(x, y) = x \land y\]
for all $x, y \in D$.

Lastly, for all $a, b, c \in \{-, 0, +\}$ we define $f_{abc}$ to be the symmetric binary operation
\begin{center}
    \binarythree{a}{b}{c},
\end{center}
and for $a, b, c, d \in \{0, 1, 2, 3\}$ we define $f_{abcdef}$ to be the symmetric binary operation
\begin{center}
    \binaryfour{a}{b}{c}{d}{e}{f}.
\end{center}

%% file: goals.tex
Recall that our goal is to verify Conjecture \ref{conj} for all clones over a domain of size 4 or less; that is, our goal is to prove that all semi-round clones over a domain of size at most 4 are round.

The relevance of this problem to CSPs has been demonstrated in \cite{kun}. The authors prove that a CSP is solved by its linear programming relaxation if and only if the clone of all operations preserving the relations defining the CSP is round.

A related result about cyclic operations is proved in \cite{barto}. They prove that every Taylor algebra contains a cyclic operation of every prime arity greater than the size of its domain.

\bigskip

It turns out that the $\size{D}$ bound is tight when $\size{D}$ is prime.

\begin{proposition}\label{bound}
    Let $D$ be a domain with prime cardinality. Then there exists a clone $\mathcal O$ over $D$ containing symmetric operations of arities $1, 2, \ldots, \lvert D \rvert - 1$ that is not round.
\end{proposition}

\begin{proof}
By construction, we can force $\mathcal O$ to be compatible a cyclic automorphism. Let $p$ be the cardinality of the domain, let $D = \ZZ / p\ZZ$, and define the $k$-ary operation $f_k$ for $k \in \{1, 2, \ldots, \lvert D \rvert - 1\}$ to be any symmetric operation satisfying
\[f_k(x_1+c, \ldots, x_k+c) = f_k(x_1, \ldots, x_k)+c\]
for all $(x_1, \ldots, x_k) \in D^k$ and $c \in D$. For instance, we can take
\[
f_k(x_1, ..., x_k) \coloneqq \frac{x_1 + \cdots + x_k}{k} \pmod{p}.
\]
Then
\[\gen{f_1, f_2, \ldots, f_{p - 1}}\]
is compatible with a cyclic automorphism, so no $p$-ary symmetric operation $f_{p}$ can exist, as
\[f_{p}\left(x_1, \ldots, x_{p}\right)\]
cannot be preserved by the automorphism.
\end{proof}

\begin{remark}[Lemma 4 of \cite{carvalho}]\label{f4}
The theorem statement is false if $\lvert D \rvert$ is not required to be prime; the construction in \cite{carvalho} proves that a 4-ary symmetric operation $f_4(w, x, y, z)$ can be constructed from a binary symmetric operation $f_2(x, y)$ and a ternary symmetric operation $f_3(x, y, z)$:
\begin{align*}
    f_4(w, x, y, z) = f_3(&f_2(f_2(w, x), f_2(y, z)), \\
    &f_2(f_2(w, y), f_2(x, z)), \\
    &f_2(f_2(w, z), f_2(x, y)))
\end{align*}
is symmetric.
\end{remark}

%% file: results.tex
These results will help us classify the semi-round clones over a domain of size at most four, up to a renaming of the domain elements.

\begin{theorem}\label{minimality}
    Let $\Sigma$ be a set of functional equations. Then every clone $\fO$ containing operations that simultaneously satisfy each functional equation in $\Sigma$ contains a minimal subclone containing operations that simultaneously satisfy each functional equation in $\Sigma$.
\end{theorem}

\begin{proof}
    For a set $\Sigma_1$ of functional equations, let property $P_{\Sigma_1}$ of a clone $\fO$ denote the assertion that $\fO$ contains a set of operations simultaneously satisfying each equation in $\Sigma_1$.
    
    For any finite subset $\Sigma' \subseteq \Sigma$, there can only be finitely many combinations of operations in $\fO$ that satisfy $\Sigma$, since the domain is finite.
    Therefore, the intersection of every chain \[\fO_1 \supsetneq \fO_2 \supsetneq \ldots\] of clones with property $P_{\Sigma'}$ also has property $P_{\Sigma'}$ for every finite subset $\Sigma' \subseteq \Sigma$. By the logical compactness theorem, $\fO$ has property $P_\Sigma$ if $\fO$ has property $P_{\Sigma'}$ for every finite subset $\Sigma' \subseteq \Sigma$, which allows us to conclude that the intersection of every chain \[\fO_1 \supsetneq \fO_2 \supsetneq \ldots\] of clones with property $P_{\Sigma}$ also has property $P_{\Sigma}$. Therefore, every chain in the poset of subclones of $\fO$ with property $P_\Sigma$, ordered by reverse inclusion, has an upper bound, so Zorn's lemma implies the existence of a minimal clone with property $P_\Sigma$. 
\end{proof}

\begin{corollary}\label{minround}
    Every round clone contains a minimal round subclone. 
\end{corollary}

\begin{proof}
    Let $\Sigma$ be the set of functional equations
    \begin{gather*}
        f_2(x, y) = f_2(y, x) \\
        f_3(x, y, z) = f_3(x, z, y) = f_3(y, x, z) = f_3(y, z, x) = f_3(z, x, y) = f_3(z, y, x) \\
        \vdots,
    \end{gather*}
    where each line asserts the existence of a symmetric operation of some arity. Then $P_\Sigma$ is equivalent to roundness, so there exists a minimal round subclone by Theorem \ref{minimality}.
\end{proof}

\begin{corollary}\label{minsemiround}
    Every semi-round clone contains a minimal semi-round subclone.
\end{corollary}

\begin{proof}
    This is the same as the proof of Corollary \ref{minround}, except the set of equations $\Sigma$ asserts the existence of symmetric operations of arities $1, 2, \ldots, \size{D}$, where $D$ is the domain, instead of symmetric operations of every arity.
\end{proof}

\begin{theorem}\label{idempotent}
    Suppose $\fO$ is a minimal counterexample to Conjecture \ref{conj} over the smallest possible domain. Then $\fO$ is idempotent.
\end{theorem}

\begin{proof}
    This is a corollary of a folklore result that seems to have first shown up in the context of CSPs in \cite{bulatov}, but we reproduce it here. It suffices to show that the only unary operation in a minimal counterexample $\fO$ is $\pi_1^1$. To do this, let the domain be $D$ and let $f_k$ be a $k$-ary symmetric operation in $\fO$ for each $k \in \{1, 2, \ldots, \size{D}\}$. If there is a non-identity unary operation $u(x) \in \fO$, then either some unary operation is not injective or all unary operations are permutations.
    
    If all unary operations are permutations, then one can find idempotent symmetric operations of arities $1, 2, \ldots, \size{D}$. To do this, define
    \[u_k(x) = f_k(\underbrace{x, \ldots, x}_{k \text{ times}}),\]
    which must be a permutation. Then $u_k^{N_k}(x) = x$ for some positive integer $N_k$. Define
    \[g_k(x_1, \ldots, x_k) = f_k\left(u_k^{N_k-1}(x_1), \ldots, u_k^{N_k-1}(x_k)\right)\]
    and note that $g_k$ is idempotent for each $k$. Finally,
    \[\gen{g_1, \ldots, g_k} \subsetneq \mathcal O\]
    is an idempotent semi-round proper subclone of $\fO$ which is also a counterexample to the conjecture, contradicting the minimality of $\fO$.
    
    If there is a unary operation $u(x)$ that is not injective, then $u^{2N}(x) = u^N(x)$ for some positive integer $N$, and $u^N(a) = u^N(b) = a$ for some distinct $a, b \in D$. Define
    \[g_k(x_1, \ldots, x_k) = u^N\left(f_k\left(u^N(x_1), \ldots, u^N(x_k)\right)\right)\]
    and note that each $g_k$ is a symmetric operation that acts on $a$ and $b$ the same way; that is, replacing $a$ with $b$ or $b$ with $a$ in $g_k$'s input will not change its output. Since
    \[\gen{g_1, \ldots, g_k} \subseteq \fO\]
    is effectively a clone on the domain $D \setminus \{b\}$, $\fO$ is not minimal with respect to domain size, contradiction.
\end{proof}

\begin{proposition}\label{semilatticeoperation}
    Let $f$ be a semilattice operation. Then $\gen{f}$ is round.
\end{proposition}

\begin{proof}
    Consider the poset determined by $f$. By induction, the symmetric $k$-ary operation $f_k$ that returns the greatest upper bound of its $k$ inputs is in the clone.
    
    Indeed, $f_1 = \pi_1^1$ and $f_2 = f$. Now
    \[f_{k+1}(x_1, \ldots, x_{k+1}) = f_2(f_k(x_1, \ldots, x_k), x_{k+1})\]
    by induction.
\end{proof}

\begin{theorem}\label{roundsemilattice}
    Let $\fO$ be an idempotent round clone over a domain $D$, and suppose some binary operation in $\fO$ acts like a height-1 semilattice over some subset $D' \subseteq D$. Then $\fO$ contains a round subclone that acts like a height-1 semilattice over $D'$.
\end{theorem}

\begin{proof}
    Let $t_2$ be the given binary operation, and for each $k \in \ZZ^+$, let $f_k$ denote a $k$-ary symmetric operation in $\fO$. 
    Then the binary operation
    \[g_2(x, y) = f_2(t_2(x, y), t_2(y, x))\]
    is symmetric and acts like a height-1 semilattice over $D'$. By induction, the $k$-ary operation
    \[g_k(x_1, \ldots, x_k) = f_k\left(c_{g_{k-1}}^2(x_1, \ldots, x_k)\right)\]
    is symmetric and acts like a height-1 semilattice over $D'$. (Recall that $c_{g_{k-1}}$ is the function that takes a $k$-tuple and returns the $k$-tuple whose $i^\th$ entry is $g_{k-1}$ applied on the $k-1$ variables not in the $i^\th$ coordinate, for all $1 \leq i \leq k$.) Then
    \[\gen{g_1, g_2, \ldots}\]
    is a round subclone of $\fO$ that acts like a height-1 semilattice over $D'$, as desired.
\end{proof}

\begin{corollary}\label{semiroundsemilattice}
    Let $\fO$ be an idempotent semi-round clone over a domain $D$, and suppose some binary operation in $\fO$ acts like a height-1 semilattice over some subset $D' \subseteq D$. Then $\fO$ contains a semi-round subclone that acts like a height-1 semilattice over $D'$.
\end{corollary}

\begin{proof}
    This is the essentially the same as the proof of Theorem \ref{roundsemilattice}, except the induction stops after constructing the $\size{D}$-ary symmetric operation.
\end{proof}

\begin{theorem}\label{lin}
    Let $\mathcal O$ be an idempotent clone over a domain $D$ with a binary symmetric operation $f_2$ that acts linearly on a subset $D' \subseteq D$ with odd prime cardinality, and suppose $\mathcal O$ contains a $\lvert D' \rvert$-ary symmetric operation $f_{\lvert D' \rvert}$. Then one can find symmetric operations $g_1, g_2, \ldots \in \fO$ and a constant $c$ such that
    \[g_k(x_1, \ldots, x_k) = \begin{cases} x & \text{if $\{x_1, x_2, \ldots, x_k\} = \{x\}$}\\ c & \text{otherwise}\end{cases}\]
    for all $k \in \ZZ^+$ and tuples $(x_1, x_2, \ldots, x_k) \in (D')^k$.
\end{theorem}

\begin{proof}
    We prove this by induction on $k$. Define $c := f_{\lvert D' \rvert}(d_1, \ldots, d_{\size{D'}})$ where $D' = \{d_1, \ldots, d_{\size{D'}}\}$, and for $k=1$ note $g_1 = \pi_1^1$. For the $k=2$ construction, assume $D' = \{0, 1, \ldots, p-1\}$ for some odd prime $p$.
    Since $f_2(x, y) = \frac{x+y}{2} \pmod{p}$ over $D'$,
    one can construct any operation that acts like
    \[\frac{a}{2^b}x + \frac{2^b-a}{2^b}y \pmod{p}\]
    over $D'$ by composing $f_2$ with itself for any $a, b \in \ZZ^+$ with $0 \leq a \leq 2^b$; in particular, by choosing $b={p-1}$, the operation
    \[ax - (a-1)y \pmod{p}\]
    can be constructed using only $f_2$ for any $a \in \ZZ$, by Fermat's Little Theorem. Finally, define
    \[g_2(x, y) \coloneqq f_k(x, 2x-y, 3x-2y, \ldots, px - (p-1)y) \pmod{p};\]
    this works because
    \[\{x, 2x-y, 3x-2y, \ldots, px - (p-1)y\} = \{0, 1, \ldots, p-1\} \pmod{p}\]
    whenever $x - y \neq 0$, by the primality of $p$.
    To construct $g_{k+1}$ given $g_{k}$, define
    \begin{align*}
        g_{k+1}(x_1, x_2, \ldots, x_{k+1}) \coloneqq g_2(& g_k(x_1, x_2, \ldots, x_k), \\
        & g_k(f_2(x_1, x_{k+1}), f_2(x_2, x_{k+1}), \ldots, f_2(x_k, x_{k+1}))).
    \end{align*}
    This works by casework on whether or not $x_1 = x_2 = \ldots = x_k$; if they are not all equal, then both arguments of $g_2$ are $c$, and if they are all equal, then both arguments of $g_2$ are distinct elements of $D'$.
\end{proof}

\begin{lemma}\label{nolin}
    Suppose a semi-round clone $\fO$ over a domain $D$ with $\size{D} \leq 4$ contains a binary symmetric operation that acts linearly on a three-element set $D' \subseteq D$. Then $\fO$ is not minimally semi-round.
\end{lemma}

\begin{proof}
    Let $c$ be the constant guaranteed by Theorem \ref{lin}. If $c \in D'$ then we are done by Corollary \ref{semiroundsemilattice}. Thus we can assume $D = \{0, 1, 2, 3\}$, $D' = \{0, 1, 2\}$, and $c=3$.
    
    The binary symmetric operation that acts linearly over $D'$ must be of the form
    \begin{center}
        \binaryfour{2}{1}{a}{0}{b}{c}
    \end{center}
    and $g_2$ must be of the form
    \begin{center}
        \binaryfour{3}{3}{d}{3}{e}{f}.
    \end{center}
    
    If at least two of $d, e, f$ are equal to 3, then $f_{33d3ef}$ acts like a height-1 semilattice over a subset $D' \subset D$ with size 3, so we can apply Corollary \ref{semiroundsemilattice} to show that $\fO$ is not minimally semi-round. The set $D'$ contains 3 and the corresponding domain elements for the two of $d$, $e$, $f$ that are equal to 3.
    
    If there is at most one 3 among $d$, $e$, and $f$, we can assume there is at most one 3 among $a$, $b$, and $c$ by Corollary \ref{semiroundsemilattice} on a two-element subset of $D$; otherwise $\fO$ is not minimally semi-round. If $a \in \{1, 2\}$ then one can check that
    \[\begin{bmatrix}a \\ a\end{bmatrix} \in \Sg{\fO}{\begin{bmatrix}1 \\ 2\end{bmatrix}, \begin{bmatrix}2 \\ 1\end{bmatrix}}.\]
    Thus, some binary operation acts like a height-1 semilattice over $\{1, 2\}$, so $\fO$ is not minimally semi-round. Similarly, if $b \in \{0, 2\}$ or $c \in \{0, 1\}$ then $\fO$ is not minimally semi-round.
    
    If exactly one of $a$, $b$ and $c$ is equal to 3, then without loss of generality assume $c=3$, so the binary operation that acts linearly over $D'$ is forced to be $f_{210013}:$
    \begin{center}
        \binaryfour{2}{1}{0}{0}{1}{3}.
    \end{center}
    Since we can assume $a \notin \{1, 2\}$ and $b \notin \{0, 2\}$, by Corollary \ref{semiroundsemilattice} either $\fO$ is not minimally semi-round, or the binary symmetric operation $f_{330313}$ guaranteed by Theorem \ref{lin} is in $\fO$:
    \begin{center}
        \binaryfour{3}{3}{0}{3}{1}{3}.
    \end{center}
    Then one can check that
    \[\begin{bmatrix}1 \\ 1\end{bmatrix} \in \Sg{\fO}{\begin{bmatrix}1 \\ 2\end{bmatrix}, \begin{bmatrix}2 \\ 1\end{bmatrix}}\]
    so $\fO$ is not minimally semi-round by Corollary \ref{semiroundsemilattice}.
    
    If $3 \notin \{a, b, c\}$, then the binary operation that acts linearly over $D'$ is forced to be $f_{210012}$:
    \begin{center}
        \binaryfour{2}{1}{0}{0}{1}{2}.
    \end{center}
    By Corollary \ref{semiroundsemilattice} either $\fO$ is not minimally semi-round, or the binary symmetric operation $f_{330312}$ guaranteed by Theorem \ref{lin} is in $\fO$:
    \begin{center}
        \binaryfour{3}{3}{0}{3}{1}{2}.
    \end{center}
    Then \[ f_{330312}(f_{330312}(x, y), f_{330312}(f_{330312}(x, z), f_{330312}(y, z)))\]
    is a ternary symmetric operation, so
    \[\gen{f_{330312}} \subsetneq \mathcal O\]
    is a semi-round clone by Remark \ref{f4}, and one can check that it does not contain a symmetric binary operation other than $f_{330312}$.
\end{proof}



%% file: classification.tex
In this section, we present a complete catalogue of every minimal idempotent semi-round clone over a domain of size at most 4, up to a renaming of the domain elements. Additionally, we prove that every semi-round clone over a domain of size 4 is also round. Non-idempotent clones are not considered, as Theorem \ref{idempotent} states that a minimal counterexample to the conjecture is idempotent, and non-minimal round clones are not considered, as Corollary \ref{minsemiround} guarantees the existence of a minimally semi-round subclone of every semi-round clone.

\subsection{Domain of Size 1}
There is only one clone over a domain of size 1, which is both semi-round and round.

\subsection{Domain of Size 2}
There is only one minimal semi-round clone over a domain of size 2, up to a renaming of the domain elements. To prove this, let $D = \{0, 1\}$ and let $f_2(x, y)$ be the clone's binary symmetric operation. We can, without loss of generality, assume that $f_2(x, y) = xy$:
\begin{center}
\begin{tabular}{c|cc}
$f_{2}$ & $0$ & $1$ \\ \hline
$0$     & $0$ & $0$ \\
$1$     & $0$ & $1$
\end{tabular}.
\end{center}
Then $\gen{f_2}$ is round by Proposition \ref{semilatticeoperation} because $f_2$ is a semilattice operation, so it is the unique minimal semi-round clone over $\{0, 1\}$, up to a renaming of the domain elements.

\subsection{Domain of Size 3}\label{size3}
Assume $D = \{-1, 0, 1\}$, which we abbreviate as $\{-, 0, +\}$. By a finite case check, any minimal clone must contain, up to a renaming of the domain elements, one of following symmetric binary operations:

\bigskip
\begin{center}
\begin{tabular}{c c c}
\binarythree{0}{0}{0} & \binarythree{+}{+}{0} & \binarythree{+}{0}{0} \\[1cm]
\binarythree{0}{0}{+} & \binarythree{+}{0}{-} & \binarythree{+}{-}{0} \\[1cm]
\binarythree{-}{0}{+}.
\end{tabular}
\end{center}
\bigskip
\noindent To determine the minimal round clones, we casework on the symmetric binary operation.

\begin{itemize}
    \item $\gen{f_{000}}$ is minimally round by Proposition \ref{semilatticeoperation}, since $f_{000}$ is a semilattice operation.
    \item $\gen{f_{++0}}$ is minimally round by Proposition \ref{semilatticeoperation}, since $f_{++0}$ is a semilattice operation.
    \item $\gen{f_{++0}} \subseteq \gen{f_{+00}}$, so $f_{+00}$ does not need to be considered, since \[f_{+00}(f_{+00}(x, f_{+00}(x, y)), f_{+00}(y, f_{+00}(x, y))) = f_{++0}(x, y).\]
    \item $\gen{f_{++0}} \subseteq \gen{f_{00+}}$, so $f_{00+}$ does not need to be considered, since \[f_{00+}(f_{00+}(x, f_{00+}(x, y)), f_{00+}(y, f_{00+}(x, y))) = f_{++0}(x, y).\]
    \item $\gen{f_{+0-}}$ is minimally semi-round and minimally round, since all operations of the form \[f_{k}(x_1, \ldots, x_k) = \sgn(x_1 + \ldots + x_k)\] are in $\gen{f_{-+0}}$. For $k=1$ take $f_1 = \pi_1^1$, and for $k=2$ take $f_2=f_{-+0}$. To construct $f_{k+1}$ inductively, the identity \begin{align*}
        f_{k+1}(x_1, \ldots, x_{k+1}) = f_2(f_k(
            &f_k(\hspace{14pt} x_2, x_3, \ldots, x_k, x_{k+1}), \\
            &f_k(x_1, \hspace{14pt} x_3, \ldots, x_k, x_{k+1}), \\
            &f_k(x_1, x_2, \hspace{14pt}  \ldots, x_k, x_{k+1}), \\
            &\hspace{65pt}\vdots, \\
            &f_k(x_1, x_2, x_3, \ldots, \hspace{14pt} x_{k+1})), f_k(x_1, \ldots, x_k))
    \end{align*}
    will suffice.
    \item $\gen{f_{+-0}}$ is not semi-round because the existence of a symmetric ternary operation is forbidden by the automorphism sending $0$ to $+$, $+$ to $-$, and $-$ to $0$. Hence, we will determine all minimal semi-round clones of the form $\gen{f_{+-0}, g}$, where $g$ is some symmetric idempotent ternary operation.
    
    By Corollary \ref{semiroundsemilattice} there must be a symmetric ternary operation in $\fO$ of the form
    \[f_3(x, y, z) = \begin{cases}
        f_{+-0}(d_1, d_2) & \text{if $\thereexists (d_1, d_2 \in D) \: \{x, y, z\} = \{d_1, d_2\}$} \\
        c & \text{if $\{x, y, z\} = D$}
    \end{cases}\]
    for some $c \in D$; without loss of generality assume $c=0$. Now, $\gen{f_{+-0}, f_3}$ is round because each operation of the form
    \[f_k(x_1, \ldots, x_k) \coloneqq \begin{cases}
        f_{+-0}(d_1, d_2) & \text{if $\thereexists (d_1, d_2 \in D) \: \{x_1, \ldots, x_k\} = \{d_1, d_2\}$} \\
        0 & \text{otherwise}
    \end{cases}\]
    is in $\gen{f_{+-0}, f_3}$. Note that the base cases $k \in \{1, 2, 3\}$ are true, so it suffices to construct $f_k$ inductively. 
    Define $\xx = (x_1, \ldots, x_{k+1})$, and define $\yy = c_{f_k}^N(x)$, $\yy' = c_{f_k}^N(f_2(\xx, \yy))$, and $\yy'' = c_{f_k}^N(f_2(\xx, \yy'))$ for a sufficiently large integer $N$. Then
    \begin{align*}
        f_{k+1}(\xx) = f_3(\yy, \yy', \yy'')_1,
    \end{align*}
    where the subscript denotes the first element of the tuple; this works because $\yy = c_{f_k}^N(\xx)$ is a constant tuple for sufficiently large $N$, so $\{\yy, \yy', \yy''\} = \{(-, \ldots, -), (0, \ldots, 0), (+, \ldots, +)\}$ or $\yy = \yy' = \yy''$.
    
    \item No clone containing $f_{-0+}$ is minimally semi-round by Lemma \ref{nolin}, so this case does not need to be considered.
\end{itemize}

\subsection{Domain of Size 4}
The characterization of all minimal idempotent semi-round clones over a domain of size 4 will be done through
casework on the symmetric binary operation it contains. There are $4^{\binom{4}{2}} = 4096$ such operations but they fall into 192 distinct equivalence classes under isomorphism; representative elements are listed below.

\begin{align*}
&f_{000000}, f_{000001}, f_{000002}, f_{000011}, f_{000012}, f_{000013}, f_{000021}, f_{000023}, f_{000033}, f_{000111},\\ &f_{000112}, f_{000121}, f_{000122}, f_{000123}, f_{000132}, f_{000321}, f_{001000}, f_{001001}, f_{001002}, f_{001003},\\ &f_{001010}, f_{001011}, f_{001012}, f_{001013}, f_{001020}, f_{001021}, f_{001022}, f_{001023}, f_{001030}, f_{001031},\\ &f_{001032}, f_{001033}, f_{001100}, f_{001101}, f_{001102}, f_{001103}, f_{001110}, f_{001111}, f_{001112}, f_{001113},\\ &f_{001120}, f_{001121}, f_{001122}, f_{001123}, f_{001130}, f_{001131}, f_{001132}, f_{001133}, f_{001200}, f_{001201},\\ &f_{001202}, f_{001203}, f_{001210}, f_{001211}, f_{001212}, f_{001213}, f_{001220}, f_{001221}, f_{001222}, f_{001223},\\ &f_{001230}, f_{001231}, f_{001232}, f_{001233}, f_{001300}, f_{001301}, f_{001302}, f_{001303}, f_{001310}, f_{001311},\\ &f_{001312}, f_{001313}, f_{001320}, f_{001321}, f_{001322}, f_{001323}, f_{001330}, f_{001331}, f_{001332}, f_{003000},\\ &f_{003001}, f_{003002}, f_{003011}, f_{003012}, f_{003013}, f_{003021}, f_{003100}, f_{003101}, f_{003102}, f_{003110},\\ &f_{003112}, f_{003113}, f_{003120}, f_{003121}, f_{003122}, f_{003300}, f_{003301}, f_{003302}, f_{003311}, f_{003312},\\ &f_{003321}, f_{011000}, f_{011001}, f_{011002}, f_{011010}, f_{011011}, f_{011012}, f_{011013}, f_{011020}, f_{011021},\\ &f_{011022}, f_{011023}, f_{011030}, f_{011031}, f_{011032}, f_{011120}, f_{011121}, f_{011122}, f_{011123}, f_{011130},\\ &f_{011131}, f_{011132}, f_{011220}, f_{011223}, f_{011230}, f_{011231}, f_{011320}, f_{011321}, f_{011322}, f_{012000},\\ &f_{012001}, f_{012002}, f_{012003}, f_{012010}, f_{012011}, f_{012012}, f_{012013}, f_{012020}, f_{012021}, f_{012022},\\ &f_{012030}, f_{012031}, f_{012032}, f_{012100}, f_{012101}, f_{012102}, f_{012103}, f_{012120}, f_{012121}, f_{012122},\\ &f_{012130}, f_{012131}, f_{012132}, f_{012200}, f_{012203}, f_{012210}, f_{012213}, f_{012230}, f_{012300}, f_{012301},\\ &f_{012302}, f_{012310}, f_{012311}, f_{012313}, f_{012320}, f_{012321}, f_{012330}, f_{013002}, f_{013010}, f_{013011},\\ &f_{013012}, f_{013021}, f_{013022}, f_{013102}, f_{013310}, f_{013321}, f_{022101}, f_{022301}, f_{022321}, f_{023321},\\ &f_{032000}, f_{032001}, f_{032020}, f_{032021}, f_{032030}, f_{032230}, f_{032320}, f_{032321}, f_{211000}, f_{211020},\\ &f_{211300}, f_{211301}.
\end{align*}

Many of these operations do not need to be considered because the clones they generate contain other binary operations. For example, only operations in the
image of repeated composition of the map
\[f(a, b) \mapsto f(f(a, f(a, b)), f(b, f(a, b)))\]
need to be considered, which results in the following 37 operations:
\begin{align*}
&f_{000000}, f_{000002}, f_{000012}, f_{000013}, f_{000033}, f_{000111}, f_{000112}, f_{000123}, f_{000132}, f_{000321}, \\
&f_{001030}, f_{001031}, f_{001032}, f_{001033}, f_{001130}, f_{001132}, f_{001133}, f_{001230}, f_{001231}, f_{001232}, \\
&f_{001233}, f_{003012}, f_{003013}, f_{003112}, f_{003113}, f_{003312}, f_{003321}, f_{011231}, f_{011321}, f_{013310}, \\
&f_{022101}, f_{023321}, f_{032230}, f_{032320}, f_{032321}, f_{211000}, f_{211020}.
\end{align*}
Under the map
\[f(a, b) \mapsto f(f(a, f(a, f(a, b))), f(b, f(b, f(a, b))))\]
the operations $f_{001030}$, $f_{001130}$, $f_{001230}$, $f_{011321}$, $f_{013310}$ can be removed from the list. Similarly, under the map \[f(a, b) \mapsto f(f(a, f(b, f(a, b))), f(b, f(a, f(a, b))))\]
the operations $f_{211000}$ and $f_{022101}$ can be removed, and $f_{211020}$ can be removed by considering the map
\[f(a, b) \mapsto f(a, f(a, f(b, f(b, f(a, b))))).\]
Lastly, the operations $f_{000321}$, $f_{003321}$, $f_{023321}$, $f_{032230}$, $f_{032320}$, and $f_{032321}$ all act linearly over a three-element subset of their domain, so by Lemma \ref{nolin} they don't need to be considered. Therefore, only the following 23 operations need to be considered for analysis:
\begin{align*}
&f_{000000}, f_{000002}, f_{000012}, f_{000013}, f_{000033}, f_{000111}, f_{000112}, f_{000123}, f_{000132}, f_{001031}, \\ &f_{001032}, f_{001033}, f_{001132}, f_{001133}, f_{001231}, f_{001232}, f_{001233}, f_{003012}, f_{003013}, f_{003112}, \\ &f_{003113}, f_{003312}, f_{011231}.
\end{align*}

Sixteen of these operations already generate minimal round clones. The following five operations
generate round clones by Proposition \ref{semilatticeoperation} because they are semilattices:
\begin{center}
    \begin{tabular}{cc}
    \binaryfour{0}{0}{0}{0}{0}{0} & \binaryfour{0}{0}{0}{0}{0}{2} \\[1cm]
    \binaryfour{0}{0}{0}{0}{1}{2} & \binaryfour{0}{0}{0}{1}{1}{1} \\[1cm]
    \binaryfour{0}{0}{0}{1}{1}{2}. & 
    \end{tabular}
\end{center}
The following eight operations also generate round clones:
\begin{center}
    \begin{tabular}{cc}
    \binaryfour{0}{0}{0}{0}{1}{3} & \binaryfour{0}{0}{0}{0}{3}{3} \\[1cm]
    \binaryfour{0}{0}{0}{1}{2}{3} & \binaryfour{0}{0}{1}{0}{3}{1} \\[1cm]
    \binaryfour{0}{0}{1}{1}{3}{3} & \binaryfour{0}{0}{1}{2}{3}{1} \\[1cm]
    \binaryfour{0}{0}{1}{2}{3}{3} & \binaryfour{0}{1}{1}{2}{3}{1}.
    \end{tabular}
\end{center}
For each of these operations $f$, the $(k+1)$-ary operation
\begin{align*}
    f_{k+1}(x_1, \ldots, x_{k+1}) \coloneqq f_2(f_k(
        &f_k(\hspace{14pt} x_2, x_3, \ldots, x_k, x_{k+1}), \\
        &f_k(x_1, \hspace{14pt} x_3, \ldots, x_k, x_{k+1}), \\
        &f_k(x_1, x_2, \hspace{14pt}  \ldots, x_k, x_{k+1}), \\
        &\hspace{65pt}\vdots, \\
        &f_k(x_1, x_2, x_3, \ldots, \hspace{14pt} x_{k+1})), f_k(x_1, \ldots, x_k)),
\end{align*}
where $f_2 = f$, is symmetric; this is similar to the $f_{+0-}$ case from the domain of size 3 enumeration. Lastly, the following three operations generate round clones:
\begin{center}
    \begin{tabular}{cc}
    \binaryfour{0}{0}{1}{0}{3}{2} & \binaryfour{0}{0}{1}{0}{3}{3} \\[1cm]
    \binaryfour{0}{0}{1}{2}{3}{2}. & 
    \end{tabular}
\end{center}
\begin{center}\end{center}
For each of these operations, a $(k+1)$-ary symmetric operation $f_{k+1}$ can be constructed through the following induction, where \[(y_1, \ldots, y_{k+1}) = c_{f_k}^N(x_1, \ldots, x_{k+1})\] for a sufficiently large integer $N$.
\begin{align*}
    f_{k+1}(x_1, \ldots, x_{k+1}) \coloneqq f_2(f_k(
        &f_k(\hspace{14pt} y_2, y_3, \ldots, y_k, y_{k+1}), \\
        &f_k(y_1, \hspace{14pt} y_3, \ldots, y_k, y_{k+1}), \\
        &f_k(y_1, y_2, \hspace{14pt}  \ldots, y_k, y_{k+1}), \\
        &\hspace{65pt}\vdots, \\
        &f_k(y_1, y_2, y_3, \ldots, \hspace{14pt} y_{k+1})), f_k(y_1, \ldots, y_k)).
\end{align*}
This works because $(y_1, \ldots, y_{k+1})$ always lies in $(D')^{k+1}$ for some three-element subset $D' \subsetneq D$ for sufficiently large $N$:
\begin{itemize}
    \item For $\gen{f_{001032}}$, repeatedly applying $c_{f_k}$ to the tuple $(x_1, \ldots, x_{k+1})$ will always result in an element of $\{0, 1, 3\}^{k+1}$, unless $(x_1, \ldots, x_{k+1}) \in \{2, 3\}^{k+1}$.
    \item For $\gen{f_{001033}}$, repeatedly applying $c_{f_k}$ to the tuple $(x_1, \ldots, x_{k+1})$ will always result in an element of $\{0, 1, 3\}^{k+1}$, unless $(x_1, \ldots, x_{k+1}) = (2, \ldots, 2)$.
    \item For $\gen{f_{001232}}$, repeatedly applying $c_{f_k}$ to the tuple $(x_1, \ldots, x_{k+1})$ will always result in an element of $\{0, 1, 2\}^{k+1}$, unless $(x_1, \ldots, x_{k+1}) \in \{1, 3\}^{k+1}$.
\end{itemize}

The following five operations do not generate clones with a symmetric ternary operation, but they generate round clones when symmetric ternary operations are added:
\begin{center}
    \begin{tabular}{cc}
    \binaryfour{0}{0}{0}{1}{3}{2} & \binaryfour{0}{0}{3}{0}{1}{2} \\[1cm]
    \binaryfour{0}{0}{3}{0}{1}{3} & \binaryfour{0}{0}{3}{1}{1}{2} \\[1cm]
    \binaryfour{0}{0}{3}{1}{1}{3}.
    \end{tabular}
\end{center}
We casework on each one. For the remainder of this paragraph, let $f_2$ be the binary symmetric operation and let $f_3$ be the ternary symmetric operation. The induction used to prove that each case yields a round clone is similar to the $f_{+-0}$ case from the domain of size 3 enumeration. To construct the $k$-ary operation $f_{k+1}$, define $\xx \coloneqq (x_1, \ldots, x_{k+1})$, and define $\yy \coloneqq c_{f_k}^N(x)$, $\yy' \coloneqq c_{f_k}^N(f_2(\xx, \yy))$, and $\yy'' \coloneqq c_{f_k}^N(f_2(\xx, \yy'))$ for a sufficiently large integer $N$. Then
\begin{align*}
    f_{k+1}(\xx) \coloneqq f_3(\yy, \yy', \yy'')_1,
\end{align*}
where the subscript denotes the first element of the tuple, is symmetric; this works for each case because $\yy = c_{k+1}^N(\xx)$ is a constant tuple for sufficiently large $N$, so either $\yy = \yy' = \yy''$ or $\{\yy, \yy', \yy''\} = \{(d_1, \ldots, d_1), (d_2, \ldots, d_2), (d_3, \ldots, d_3)\}$, where $\{d_1, d_2, d_3\}$ is chosen such that $f_2$, when restricted to the domain $\{d_1, d_2, d_3\} \subset D$, can be renamed to $f_{+-0}$.

The function $\theta$ that takes a binary operation $f_2$ and a ternary operation $f_3$ as input and outputs a ternary operation is defined as
\[\theta(f_2, f_3)(x, y, z) \coloneqq c_{f_2}^N\left(f_2\left(c_{f_2}^N(x, y, z), f_3\left(c_{f_2}^N(x, y, z)\right)\right)\right)_1\]
for a sufficiently large positive integer $N$. In each of the following cases, $\theta$ returns a symmetric ternary operation that modifies only one or two outputs of $f_3$. The casework on the binary symmetric operation is below.

\begin{itemize}
    \item Suppose a clone is generated by $f_{000132}$ and a symmetric ternary operation. Using Corollary \ref{semiroundsemilattice}, one can force the existence of a symmetric ternary operation of the form
    \[g_c(x, y, z) = \begin{cases}
        0 & 0 \in \{x, y, z\} \\
        1 & \{x, y, z\} \in \{\{1\}, \{1, 2\}\} \\
        2 & \{x, y, z\} \in \{\{2\}, \{2, 3\}\} \\
        3 & \{x, y, z\} \in \{\{3\}, \{1, 3\}\} \\
        c & \{x, y, z\} = \{1, 2, 3\}
    \end{cases}\]
    for some $c \in D$. Since $\theta(f_{000132}, g_c)$ maps
    \[g_1 \mapsto g_3 \mapsto g_2 \mapsto g_1,\]
    this case gives two distinct minimal round clones.
    \item Suppose a clone is generated by $f_{003012}$ and a symmetric ternary operation. Using Corollary \ref{semiroundsemilattice}, one can force the existence of a symmetric ternary operation of the form
    \[g_{c,d}(x, y, z) = \begin{cases}
        0 & \{x, y, z\} \in \{\{0\}, \{0, 1\}, \{0, 2\}, \{1, 2\}, \{0, 1, 2\},\\
        & \hspace{53pt} \{1, 2, 3\}\} \\
        1 & \{x, y, z\} \in \{\{1\}, \{1, 3\}\} \\
        2 & \{x, y, z\} \in \{\{2\}, \{2, 3\}\} \\
        3 & \{x, y, z\} \in \{\{3\}, \{0, 3\}\} \\
        c & \{x, y, z\} = \{0, 1, 3\} \\
        d & \{x, y, z\} = \{0, 2, 3\}
    \end{cases}\]
    for some $c, d \in D$. Since $\theta(f_{003012}, g_c)$ maps
    \begin{align*}
        &g_{0, 0} \mapsto g_{3, 3} \mapsto g_{1, 2} \mapsto g_{0, 0} \\
        &g_{0, 2} \mapsto g_{3, 0} \mapsto g_{1, 3} \mapsto g_{0, 2} \\
        &g_{0, 3} \mapsto g_{3, 2} \mapsto g_{1, 0} \mapsto g_{0, 3}
    \end{align*}
    and maps each of $g_{0, 1}$, $g_{1, 1}$, $g_{2, 0}$, $g_{2, 1}$, $g_{2, 2}$, $g_{2, 3}$, $g_{3, 1}$ to one of the above three cycles, this case gives three distinct minimal round clones.
    \item Suppose a clone is generated by $f_{003013}$ and a symmetric ternary operation. Using Corollary \ref{semiroundsemilattice}, one can force the existence of a symmetric ternary operation of the form
    \[g_{c}(x, y, z) = \begin{cases}
        0 & \{x, y, z\} \in \{\{0\}, \{0,1\}, \{0,2\}, \{1,2\},\{0,1,2\}\}\\
        1 & \{x, y, z\} \in \{\{1\},\{1,3\}\} \\
        2 & \{x, y, z\} = \{2\} \\
        3 & \{x, y, z\} \in \{\{3\},\{0,3\},\{2,3\},\{0,2,3\}\} \\
        c & \{x, y, z\} \in \{\{0, 1, 3\}, \{1,2,3\}\}
    \end{cases}\]
    for some $c \in D$. Since $\theta(f_{003013}, g_c)$ maps
    \[g_0 \mapsto g_3 \mapsto g_1 \mapsto g_0\]
    and maps $g_2$ to the above cycle, this case only gives one minimal round clone.
    \item Suppose a clone is generated by $f_{003112}$ and a symmetric ternary operation. Using Corollary \ref{semiroundsemilattice}, one can force the existence of a symmetric ternary operation of the form
    \[g_{c,d}(x, y, z) = \begin{cases}
        0 & \{x, y, z\} \in \{\{0\}, \{0, 1\}, \{0, 2\}, \{0, 1, 2\}\}\\
        1 & \{x, y, z\} \in \{\{1\}, \{1, 2\}, \{1, 3\}, \{1, 2, 3\}\} \\
        2 & \{x, y, z\} \in \{\{2\}, \{2, 3\}\} \\
        3 & \{x, y, z\} \in \{\{3\}, \{0, 3\}\} \\
        c & \{x, y, z\} = \{0, 1, 3\} \\
        d & \{x, y, z\} = \{0, 2, 3\}
    \end{cases}\]
    for some $c, d \in D$. Since $\theta(f_{003012}, g_c)$ maps
    \begin{align*}
        &g_{0, 0} \mapsto g_{3, 3} \mapsto g_{1, 2} \mapsto g_{0, 0} \\
        &g_{0, 2} \mapsto g_{3, 0} \mapsto g_{1, 3} \mapsto g_{0, 2} \\
        &g_{0, 3} \mapsto g_{3, 2} \mapsto g_{1, 0} \mapsto g_{0, 3}
    \end{align*}
    and maps each of $g_{0, 1}$, $g_{1, 1}$, $g_{2, 0}$, $g_{2, 1}$, $g_{2, 2}$, $g_{2, 3}$, $g_{3, 1}$ to one of the above three cycles, this case gives three distinct minimal round clones.
    \item Suppose a clone is generated by $f_{003113}$ and a symmetric ternary operation. Using Corollary \ref{semiroundsemilattice}, one can force the existence of a symmetric ternary operation of the form
    \[g_{c}(x, y, z) = \begin{cases}
        0 & \{x, y, z\} \in \{\{0\}, \{0, 1\}, \{0, 2\}, \{0, 1, 2\}\}\\
        1 & \{x, y, z\} \in \{\{1\}, \{1, 2\}, \{1, 3\}, \{1, 2, 3\}\} \\
        2 & \{x, y, z\} = \{2\} \\
        3 & \{x, y, z\} \in \{\{3\}, \{0, 3\}, \{2, 3\}, \{0, 2, 3\}\} \\
        c & \{x, y, z\} = \{0, 1, 3\}
    \end{cases}\]
    for some $c \in D$. Since $\theta(f_{003013}, g_c)$ maps
    \[g_0 \mapsto g_3 \mapsto g_1 \mapsto g_0\]
    and $g_2$ to the above cycle, this case gives two distinct minimal round clones.
\end{itemize}
One can prove that each of the clones enumerated above are distinct by computing relations; for each pair of distinct clones $\fO_1$ and $\fO_2$ in the above list, one can find a relation that is preserved by $\fO_1$ but not by $\fO_2$.

\bigskip

The following two operations also don't generate clones with a symmetric ternary operation, but they generate round clones when symmetric ternary operations are added:
\begin{center}
    \begin{tabular}{cc}
    \binaryfour{0}{0}{1}{1}{3}{2} & \binaryfour{0}{0}{3}{3}{1}{2}
    \end{tabular}.
\end{center}
For the remainder of this paragraph, let $f_2$ be the binary symmetric operation and let $f_3$ be the ternary symmetric operation. To construct the $k$-ary operation $f_{k+1}$ for each case, define:
\begin{align*}
    \xx &\coloneqq (x_1, \ldots, x_{k+1}) \\
    \yy &\coloneqq c_{f_k}^N\left(f_2\left(c_{f_k}^N(\xx), c_{f_k}^{N+1}(\xx)\right)\right) \\
    \yy' &\coloneqq c_{f_k}^N\left(f_2\left(c_{f_k}^N(f_2(\xx, \yy)), c_{f_k}^{N+1}(f_2(\xx, \yy))\right)\right) \\
    \yy'' &\coloneqq c_{f_k}^N\left(f_2\left(c_{f_k}^N(f_2(\xx, \yy')), c_{f_k}^{N+1}(f_2(\xx, \yy'))\right)\right)
\end{align*}
for a sufficiently large integer $N$. Then
\begin{align*}
    f_{k+1}(\xx) \coloneqq f_3(\yy, \yy', \yy'')_1,
\end{align*}
where the subscript denotes the first element of the tuple, is symmetric; this works for each case because $\yy = c_{f_k}^N\left(f_2\left(c_{f_k}^N(\xx), c_{f_k}^{N+1}(\xx)\right)\right)$ is a constant tuple for sufficiently large $N$, so either $\yy = \yy' = \yy''$ or $\{\yy, \yy', \yy''\} = \{(d_1, \ldots, d_1), (d_2, \ldots, d_2), (d_3, \ldots, d_3)\}$, where $f_2$ restricted to the domain $\{d_1, d_2, d_3\} \subset D$ can be renamed to $f_{+-0}$.

The function $\Theta$ that takes a binary operation $f_2$ and a ternary operation $f_3$ as input and outputs a ternary operation is defined as
\begin{align*}
    \Theta(f_2, f_3)(x, y, z) \coloneqq c_{f_2}^N\biggl(&f_2\biggl(f_2(c_{f_2}^N(x, y, z), c_{f_2}^{N+1}(x, y, z)), \\
    &f_3\left(f_2(c_{f_2}^N(x, y, z), c_{f_2}^{N+1}(x, y, z))\right)\biggr)\biggr)_1
\end{align*}
for a sufficiently large positive integer $N$. In each of the following cases, $\Theta$ returns a symmetric ternary operation that modifies only one or two outputs of $f_3$.
\begin{itemize}
    \item Suppose a clone is generated by $f_{001132}$ and a symmetric ternary operation. Using Corollary \ref{semiroundsemilattice}, one can force the existence of a symmetric ternary operation of the form
    \[g_{c}(x, y, z) = \begin{cases}
        0 & \{x, y, z\} \in \{\{0, 2\}, \{0, 1, 2\}, \{0, 2, 3\}\}\\
        1 & \{x, y, z\} = \{1, 2\} \\
        2 & \{x, y, z\} \in \{\{2\}, \{2, 3\}\} \\
        \sigma^{-1}(\sgn(\sigma(x)+\sigma(y)+\sigma(z))) & \{x, y, z\} \subseteq \{0, 1, 3\} \\
        c & \{x, y, z\} = \{1, 2, 3\} 
    \end{cases}\]
    for some $c \in D$, where $\sigma(0) = -1$, $\sigma(1) = 0$, and $\sigma(3) = 1$. This case only gives one minimal round clone, since $\Theta(f_{001132}, g_c)$ maps
    \[g_1 \mapsto g_3 \mapsto g_2 \mapsto g_1.\]
    To prove that $\gen{f_{001132}, g_0} \subseteq \gen{f_{001132}, g_1}$, let $\xx \coloneqq (x, y, z)$ and $\yy \coloneqq f_2\left(c_{f_{001132}}^N(\xx), c_{f_{001132}}^{N+1}(\xx)\right)$, which will be constant unless $\{x, y, z\} = \{1, 2, 3\}$. Additionally, let $\yy' \coloneqq c_{f_{001132}}(\yy)$ and $\yy'' \coloneqq c_{f_{001132}}(\yy')$ be the cyclic rotations of $\yy$. Then
    \[c_{f_{003312}}^N(f_2(f_2(f_2((g_0(\yy), g_0(\yy), g_0(\yy)), \yy), \yy'), \yy''))_1 = g_1(\xx)\]
    where the subscript denotes the first element of the tuple, as desired.
    \item Suppose a clone is generated by $f_{003312}$ and a symmetric ternary operation. Using Corollary \ref{semiroundsemilattice}, one can force the existence of a symmetric ternary operation of the form
    \[g_{c, d}(x, y, z) = \begin{cases}
        0 & \{x, y, z\} \in \{\{0\}, \{0, 1\}, \{0, 2\}\}\\
        3 & \{x, y, z\} \in \{\{0, 3\}, \{0, 1, 2\}\} \\
        \sigma^{-1}(\sgn(\sigma(x)+\sigma(y)+\sigma(z))) & \{x, y, z\} \subseteq \{1, 2, 3\} \\
        c & \{x, y, z\} = \{0, 1, 3\} \\
        d & \{x, y, z\} = \{0, 2, 3\}
    \end{cases}\]
    for some $c, d \in D$, where $\sigma(1) = -1$, $\sigma(2) = 1$, and $\sigma(3) = 0$. This case gives three minimal round clones, since $\Theta(f_{003312}, g_{c,d})$ maps
    \begin{align*}
        &g_{0, 0} \mapsto g_{3, 3} \mapsto g_{1, 2} \mapsto g_{0, 0} \\
        &g_{0, 2} \mapsto g_{3, 0} \mapsto g_{1, 3} \mapsto g_{0, 2} \\
        &g_{0, 3} \mapsto g_{3, 2} \mapsto g_{1, 0} \mapsto g_{0, 3}
    \end{align*}
    and also eventually maps all others to the above cycles, since $g_{1, 1} \mapsto g_{0, 1} \mapsto g_{3, 0}$, $g_{2, 2} \mapsto g_{2, 0} \mapsto g_{0, 3}$, $g_{2, 1} \mapsto g_{1, 2}$, $g_{2, 3} \mapsto g_{3, 2}$, and $g_{3, 1} \mapsto g_{1, 3}$.
\end{itemize}

\noindent Since we have exhausted all cases, we have established the proof of Theorem \ref{result}.

\subsection{Domain of Size 5}
With computer assistance, it has been shown that every idempotent clone over a domain of size 5 that contains symmetric operations of arities $1$, $2$, $3$, $4$, and $5$ contains symmetric operations of arities up to 20. The code used to verify this is available on Github at \href{https://github.com/The-Turtle/PRIMES}{\texttt{https://github.com/The-Turtle/PRIMES}}.

%% file: futurework.tex
If we want to make progress on larger domains, we need a way to determine whether or not a clone has a symmetric operation of a given arity without explicitly generating one.

\begin{definition}\label{symdef} Let $\AA$ be an algebra with underlying set $A$. For any tuple $\aa = (a_1, \ldots, a_k) \in A^k$, define the \textit{symmetric relation} on $\aa$ to be the set
\[\Sym(\aa) \coloneqq
\operatorname{Sg}_{\AA^{k!}}\left(
\begin{bmatrix}
a_{\sigma_1(1)} \\ \vdots \\
a_{\sigma_{k!}(1)}
\end{bmatrix},
\ldots, 
\begin{bmatrix}
a_{\sigma_1(k)} \\ \vdots \\
a_{\sigma_{k!}(k)}
\end{bmatrix}
\right),\]
where $\sigma_1, \sigma_2, \ldots, \sigma_{k!}$ are the $k!$ permutations of the tuple $(1, \ldots, k)$.
\end{definition}

\begin{proposition} Let $\AA$ be an algebra with underlying set $A$, and suppose that for every $j \le k$ and every tuple $\aa \in A^j$, the symmetric relation $\Sym(\aa)$ contains a constant tuple. Then $\AA$ has a symmetric operation of every arity less than or equal to $k$.
\end{proposition}
\begin{proof} We prove this by induction on $k$; for the base case $k=1$, take $f_1 = \pi_1^1$. By the inductive hypothesis, there are symmetric operations $f_1, f_2, \ldots, f_{k-1}$ of every arity strictly less than $k$. Now suppose that $f$ is a $k$-ary operation such that the set $T \subseteq A^k$ of tuples for which $f$ acts symmetrically on is maximal. We claim that $T$ must equal $A^k$; to prove this, it suffices to show that if $t \in A^k \setminus T$ is a tuple which $f$ does not act symmetrically on, then there is a $k$-ary operation $g$ which acts symmetrically on $T \cup \{t\}$.

We will first construct, for each $j < k$, an operation $g_j$ which acts symmetrically on $T$ and which is unchanged by every permutation of its first $j$ variables. We start by taking $g_1 = f$, and then we inductively define $g_j$ as
\begin{align*}
    g_j(x_1, \ldots, x_k) \coloneqq
    f_j(&g_{j-1}(x_1, x_2, \ldots, x_{j-1}, x_j, \hspace{9pt} x_{j+1}, \ldots, x_k), \\
    & g_{j-1}(x_2, x_3, \ldots, x_j, \hspace{10pt} x_1, \hspace{9pt} x_{j+1}, \ldots, x_k),\\
    &\hspace{60pt}\vdots, \\
    &g_{j-1}(x_j, x_1, \ldots, x_{j-2}, x_{j-1}, x_{j+1}, \ldots, x_k)).
\end{align*}
Finally, let $\aa$ be the tuple
\[
\aa \coloneqq (g_{k-1}(t_1, \ldots, t_k), g_{k-1}(t_2, \ldots, t_k, t_1), \ldots, g_{k-1}(t_k, t_1, \ldots, t_{k-1})).
\]
By assumption, $\Sym(\aa)$ contains a constant tuple, so there must be some $k$-ary operation $h \in \operatorname{Clo}(\AA)$ which acts symmetrically on $\aa$. Then we define $g$ by
\[
g(x_1, \ldots, x_k) \coloneqq h(g_{k-1}(x_1, \ldots, x_k), g_{k-1}(x_2, \ldots, x_k, x_1), \ldots, g_{k-1}(x_k, x_1, \ldots, x_{k-1})).\qedhere
\]
\end{proof}

The relation $\Sym(\aa)$ has a useful special property.

\begin{proposition}\label{baby-reversible} Let $\AA$ be an algebra with underlying set $A$. For any tuple $\aa \in A^k$ and any pair of permutations $(\sigma, \tau)$ on $\{1, \ldots, k\}$, let $\PP \le \AA^2$ be the binary relation $\pi_{i_\sigma,i_\tau}(\Sym(\aa))$, where $i_\sigma$ and $i_\tau$ are the indices of $\sigma$ and $\tau$ as defined in Definition \ref{symdef}. Then for any subset $B \subseteq A$, we have
\[
B + \PP = B \implies B - \PP = B.\]
\end{proposition}
\begin{proof} Define
\[\PP^{\circ n} \coloneqq \underbrace{\PP \circ \cdots \circ \PP}_{\text{$n$ $\PP$'s}}.\]
Then we have $\PP^- \subseteq \PP^{\circ (k! - 1)}$, since $\PP^{\circ (k! - 1)}$ contains the generators of $\PP^-$, so $B - \PP \subseteq B + \PP^{\circ (k! - 1)} = B$.
Similarly, $B = B+\PP \subseteq B-\PP^{\circ k!-1} \subseteq B$. Hence $B - \PP$ must in fact equal $B$.
\end{proof}

\begin{definition} Let $\AA$ be an algebra with underlying set $A$. Say that a relation $\RR \le \AA^m$ is \emph{reversible} if it satisfies the following two properties:
\begin{itemize}
\item for all $i,j$ we have $\pi_i(\RR) = \pi_j(\RR)$, and
\item for every sequence $p = ((i_1,j_1), ..., (i_k, j_k))$ of pairs of coordinates of $\RR$, if we define the binary relation $\PP_p \le \AA^2$ by
\[
\PP_p \coloneqq \pi_{i_1,j_1}(\RR) \circ \cdots \circ \pi_{i_k,j_k}(\RR),
\]
then for every $B \subseteq A$, we have
\[
B + \PP_p = B \implies B - \PP_p = B.
\]
\end{itemize}
\end{definition}

\begin{proposition} For every algebra $\AA$ with underlying set $A$ and every tuple $\aa \in A^n$, the relation $\Sym(\aa)$ is reversible.
\end{proposition}
\begin{proof} Since the marginal distributions of each coordinate of the uniform distribution on the set of tuples in $\Sym(\aa)$  are equal, this follows from the implication $(e) \implies (a)$ of Proposition \ref{bin-reversible} below. 
\end{proof}

We have the following strong refinement of our main conjecture.

\begin{conjecture}\label{reversible-conj} Suppose that $\AA$ is a finite idempotent algebra, such that for every subquotient $\BB \in HS(\AA)$ there is some element $b \in \BB$ which is fixed by every automorphism of $\BB$. Then every reversible relation $\RR \le \AA^n$ contains a constant tuple.
\end{conjecture}


The condition involving arbitrary compositions of two-variable projections of the relation $\RR$ in the definition of reversibility is necessary, as demonstrated by the following example.

\begin{example} Let $\AA = (\{-,0,+\},\sgn(x+y))$ and let $\RR \le \AA^5$ be the relation
\[
\RR \coloneqq \left\{(x_1, x_2, x_3, x_4, x_5) \in \AA^5 \mid x_1 + x_2 + x_3 \ge 1 \wedge x_4 = -x_5\right\}.
\]
Then every binary projection of $\RR$ is reversible, but the relation $\RR$ is \emph{not} reversible: we have
\[
\{-\} + \pi_{1,2}(\RR) + \pi_{4,5}(\RR) = \{-\}
\]
but
\[
\{-\} - \pi_{4,5}(\RR) - \pi_{1,2}(\RR) = \{-,0,+\}.
\] Since $\RR$ does not contain any of the constant tuples $(-,\ldots,-)$, $(0,\ldots,0)$, or $(+, \ldots, +)$, we need the stronger condition about arbitrary compositions of two-variable projections.
\end{example}

For binary relations, the concept of reversibility simplifies.

\begin{proposition}\label{bin-reversible} If $\RR \le_{sd} \AA^2$ is a binary subdirect relation on a finite algebra $\AA$ with underlying set $A$, then the following are equivalent.
\begin{itemize}
\item[(a)] For every $B \subseteq A$, we have
\[
B + \RR = B \implies B - \RR = B.
\]

\item[(b)] If we consider the ordered pairs of $\RR$ as the edges of a directed graph $G$ with vertex set $A$, then every weakly connected component of $G$ is also strongly connected.

\item[(c)] If we consider the ordered pairs of $\RR$ as the edges of a directed graph $G$ with vertex set $A$, then every directed edge of $G$ is contained in a directed cycle of $G$.

\item[(d)] There is some $n \ge 1$ such that $\RR^- \subseteq \RR^{\circ n}$.

\item[(e)] There is a positive probability distribution with support $\RR$ such that the marginal distributions on the first and second coordinates agree.

\item[(f)] The binary relation $\RR$ is reversible; that is, every binary relation which can be written as a composition of copies of $\RR$ and $\RR^-$ satisfies (a).
\end{itemize}
\end{proposition}
\begin{proof} $(a) \implies (b)$: define a quasiorder $\preceq$ on $A$ by $a \preceq b$ if there is any $k \ge 0$ such that $(a,b) \in \RR^{\circ k}$. For any $a \in A$, there is a $\preceq$-maximal element $b \in A$ such that $a \preceq b$, by the finiteness of $A$. Let $B$ be the set of all $b'$ such that $b \preceq b'$, then the $\preceq$-maximality of $b$ implies that $B$ is a strongly connected component of $\RR$ and that $B + \RR = B$. Then $(a)$ implies that we have $B - \RR = B$, so we have
\[
a \in \{b\} - \RR^{\circ k} \subseteq B - \RR^{\circ k} = B,
\]
and similarly any element in the weakly connected component containing $a$ is also contained in $B$.

$(b) \implies (c)$ is obvious. For $(c) \implies (d)$, pick for each directed edge of $\RR$ a directed cycle containing it, and choose $n$ such that $n+1$ is a common multiple of the lengths of all of these directed cycles. $(d) \implies (a)$ and $(f) \implies (a)$ are also obvious.

To prove that $(c) \implies (e)$, find a collection $\mathcal{C}$ of directed cycles of $\RR$ that contains every edge of $\RR$ at least once. Define a probability distribution on $\RR$ by the following two step process: first pick a uniformly random cycle $C \in \mathcal{C}$, then pick a uniformly random edge $(x,y) \in C$.

For $(e) \implies (a)$, let $p_{(a,b)} > 0$ be the probability assigned to a given element $(a,b) \in \RR$ (and set $p_{(a,b)} = 0$ for $(a,b) \not\in \RR$), and let
\[
p_b \coloneqq \sum_{a \in A} p_{(a,b)} = \sum_{c\in A} p_{(b,c)}
\]
be the marginal probability of seeing $b$ on either the first or second coordinate. For any subset $B \subseteq A$, define $p(B)$ by
\[
p(B) \coloneqq \sum_{b \in B} p_b.
\]
Then we have
\[
p(B + \RR) = \sum_{b \in B+\RR} p_{b} = \sum_{b \in B+\RR} \sum_{a \in \{b\}-\RR} p_{(a,b)} \ge \sum_{b \in B+\RR} \sum_{a \in B} p_{(a,b)} = \sum_{a \in B} p_{a} = p(B),
\]
with equality only when every element $b \in B + \RR$ has $\{b\} - \RR \subseteq B$. If $B + \RR = B$, then we must have equality above, so $B - \RR = B + \RR - \RR = B$.

Given the equivalence between $(a)$ and $(e)$, $(e) \implies (f)$ follows from the fact that if $\RR$ and $\SS$ are any pair of binary relations such that there are positive probability distributions $p$ and $q$ supported on $\RR$ and $\SS$, respectively, such that the marginal of $p$ on the second coordinate equals the marginal of $q$ on the first coordinate, then there is a positive probability distribution ``$p \circ q$'' supported on $\RR \circ \SS$ such that the marginals of $p$ and $p \circ q$ on the first coordinate are equal, and the marginals of $p \circ q$ and $q$ on the second coordinate are equal.

The equivalence of $(f)$ can also be shown by proving $(d) \implies (f)$. Let $\RR'$ be a composition of $i$ copies of $\RR$ and $j$ copies of $\RR^-$, in some order;
it suffices to show that $\RR'$ satisfies $(d)$.
If $i > j$, then $\RR' \supseteq \RR^{\circ (i-j)}$, so if $\RR^- \subseteq \RR^{\circ n}$, then
\[
(\RR^-)^{\circ (i-j)} \subseteq \RR^{\circ n(i-j)} \subseteq \RR'^{\circ n},
\]
and we can finish since $\RR'$ and $\RR'^-$ are each contained in some composition of $\RR^{\circ (i-j)}$ and $(\RR^-)^{\circ (i-j)}$. The case $i < j$ is similar, so we are left with the case $i = j$.

To deal with the case $i = j$, the case where $\RR'$ is a composition of an equal number of copies of $\RR$ and $\RR^-$ in some order, code the sequence of copies of $\RR$ and $\RR^-$ as a sequence of $i$ copies of $+$ and $i$ copies of $-$. Let $a$ and $-b$ be the largest value and smallest value, respectively, of the partial sums of the sequence of $+$'s and $-$'s. Then it's easy to see that $\RR'$ contains the relations
\[
\RR_{\pm a} \coloneqq \RR^{\circ a} \circ (\RR^-)^{\circ a}
\]
and
\[
\RR_{\mp b} \coloneqq (\RR^-)^{\circ b} \circ \RR^{\circ b}.
\]
Thus, both $\RR'$ and $(\RR')^-$ are contained in some composition of copies of $\RR_{\pm a}$ and $\RR_{\mp b}$, as desired.
\end{proof}

\begin{theorem} Conjecture \ref{reversible-conj} is true for the algebra $\AA = (\{-,0,+\},\sgn(x+y))$.
\end{theorem}
\begin{proof} Let $\AA$ have underlying set $A$, and let $\RR \le \AA^n$ be a reversible relation. If $\pi_i(\RR) \ne \AA$, then $\pi_i(\RR)$ is a semilattice - we leave this case to the reader. We are left with the case $\pi_i(\RR) = A$ for all $i$; that is, the case where $\RR$ is subdirect.

A brute force enumeration shows that every binary subdirect relation on $\AA$ is one of the seven relations
\[
\{(x,y) \in A^2 \mid x = y\}, \{(x,y) \in A^2 \mid x = -y\}, \{(x,y) \in A^2 \mid x \le y\},\]\[ \{(x,y) \in A^2 \mid x \ge y\}, \{(x,y) \in A^2 \mid x + y \ge 0\}, \{(x,y) \in A^2 \mid x + y \le 0\}, A^2.
\]
In particular, each binary subdirect relation $\SS \le_{sd} \AA^2$ is completely determined by the intersection $\SS \cap \{-,+\}^2$; in fact, the composition of any pair of binary subdirect relations on $\AA$ is also determined by the composition of their restrictions to $\{-,+\}$.

Among these seven relations, the two binary relations $\{(x,y) \in A^2 \mid x \le y\}$ and $\{(x,y) \in A^2 \mid x \ge y\}$ are \emph{not} reversible. Since
\[
\{(x,y) \in A^2 \mid x = -y\} \circ \{(y,z) \in A^2 \mid y + z \ge 0\} = \{(x,z) \in A^2 \mid x \le z\}
\]
and
\[
\{(x,y) \in A^2 \mid x + y \le 0\} \circ \{(y,z) \in A^2 \mid y + z \ge 0\} = \{(x,z) \mid x \le z\},
\]
we see that every reversible subdirect arity-$k$ relation $\RR$ either
\begin{itemize}
    \item[(a)] has
    $
    \pi_{i,j}(\RR) \in \left\{\{(x, y) \in A^2 \mid x=y\}, \{(x, y) \in A^2 \mid x = -y\}, A^2\right\}
    $
    for all integers $1 \leq i,j \leq k$,
    \item[(b)] has
    $
    \pi_{i,j}(\RR) \in \left\{\{(x, y) \in A^2 \mid x=y\}, \{(x, y) \in A^2 \mid x + y \ge 0\}, A^2\right\}
    $
    for all integers $1 \leq i,j \leq k$, or
    \item[(c)] has
    $
    \pi_{i,j}(\RR) \in \left\{\{(x, y) \in A^2 \mid x=y\}, \{(x, y) \in A^2 \mid x + y \le 0\}, A^2\right\} 
    $
    for all integers $1 \leq i,j \leq k$. 
\end{itemize}
We will show that in case $(a)$, we have $(0,\ldots,0) \in \RR$, in case $(b)$ we have $(+, \ldots, +) \in \RR$, and in case $(c)$ we have $(-, \ldots, -) \in \RR$. By symmetry, we only have to consider cases $(a)$ and $(b)$. Case $(b)$ follows from the following claim.

\begin{claim}
For any relation $\SS \le \AA^n$ such that $(+,+) \in \pi_{i,j}(\SS)$ for all $i,j$, we have $(+,\ldots,+) \in \SS$.
\end{claim}

\begin{claimproof}
We will prove, by induction on $\size{I}$ that for every subset $I \subseteq \{1, 2, \ldots, n\}$ there is a tuple $s_I \in \SS$ such that its $i^\th$ coordinate is $+$ for all $i \in I$. The base case $\size{I} \le 2$ is our assumption on $\SS$. For the inductive step $|I| \ge 3$, let $i$,$j$, and $k$ be any three distinct elements of $I$. Then we define $s_I$ inductively by
\[
s_I \coloneqq \sgn\left(s_{I\setminus \{i\}},s_{I\setminus\{j\}},s_{I\setminus\{k\}}\right),
\]
using the fact that the three-variable operation $\sgn(x+y+z)$ is in the clone generated by the two-variable operation $\sgn(x+y)$, as proven in the $\gen{f_{+0-}}$ case of section \ref{size3}.
\end{claimproof}

\noindent Case $(a)$ also follows from the claim. To see this, find a maximal subset $I \subseteq \{1, 2, \ldots, n\}$ such that no pair of indices $i,j \in I$ has $\pi_{i,j}(\RR) = \{(x, y) \in A^2 \mid x = -y\}$. Then we can use the claim to show that the tuple $s$ given by
\[
\pi_i(s) = \begin{cases}+ & i \in I\\ - & i \not\in I,\end{cases}
\]
is in $\RR$. By symmetry, $-s \in \RR$ as well. Therefore $\operatorname{sgn}((s)+(-s)) = (0,\ldots,0) \in \RR$, so we are done.
\end{proof}

Using some stronger background theory, we can confirm that Conjecture \ref{reversible-conj} is true for binary relations.


\begin{theorem} Conjecture \ref{reversible-conj} holds for binary relations: if every subquotient of a finite idempotent algebra $\AA$ has an element fixed by its automorphism group, then every binary reversible relation $\RR \le \AA^2$ contains a constant tuple.
\end{theorem}
\begin{proof} For idempotent algebras, the assumption implies that $\AA$ is Taylor by Proposition 4.14 of \cite{bulatov-jeavons-varieties}; in fact, a more general form of this result is proved in Proposition 2.1 of \cite{subquotient}.

Let $A$ and $B$ be the underlying sets of $\AA$ and $\BB$, respectively. Assume without loss of generality that $\RR$ is subdirect; that is, $\pi_1(\RR) = \pi_2(\RR) = A$. Let $\theta$ be the limit of the linking congruence of the binary relation $\RR^{\circ m}$ when $m$ gets large. An alternative way to describe $\theta$ is as follows: consider $\RR$ to be the edges of a directed graph on $\AA$, and consider two vertices to be equivalent if there is an undirected path connecting them such that the total number of forward edges along the path equals the total number of backward edges along the path. Then $\RR/\theta$ is the graph of an automorphism of $\AA/\theta$, so by assumption there is a congruence class $B$ of $\theta$ which is fixed by this automorphism. Since $\AA$ is idempotent, $\BB$ is a subalgebra of $\AA$, and since $B$ is fixed by this automorphism of $\AA/\theta$, $B + \RR = B$. The fact that $B$ is a congruence class of $\theta$ is equivalent to the restriction of $\RR$ to $B$ defining a directed graph of ``algebraic length 1,'' so we can apply the Loop Lemma of \cite{barto} to conclude that $\RR$ contains a constant tuple $(b,b)$ with $b \in B$.
\end{proof}

%% file: acknowledgements.tex
We would like to thank the MIT-PRIMES program --- including Dr.\ Tanya Khovanova, Dr.\ Alexander Vitanov, Dr.\ Slava Gerovitch, and Prof.\ Pavel Etingof --- for giving us the resources to make this research possible.

%% file: appendix.tex

The linear programming relaxation of a CSP is closely related to certain weak local consistency checking procedures. The most basic form of local consistency is known as arc-consistency.

\begin{definition} An instance $I$ of a CSP with domain $A$ is called \emph{arc-consistent} if there is a way to associate to each variable $x$ of $I$ a subset $A_x \subseteq A$ such that for every constraint relation $R$ of the instance which involves the variable $x$, the set of possible values of $x$ which are compatible with the relation $R$ is exactly $A_x$.
\end{definition}

Originally it was believed that a CSP was solved by the basic linear programming relaxation if and only if every arc-consistent instance had a solution -- this claim appeared in \cite{kun}, but the proof was flawed; the CSP defined by the three element algebra $(\{-,0,+\}, \sgn(x+y))$ is a counterexample.

A stronger form of local consistency was introduced in \cite{weakprague}.

\begin{definition}\label{defn-weak-prague} An instance $I$ of a CSP is called a \emph{weak Prague instance} if it satisfies the following three conditions.
\begin{itemize}
\item[(P1)] The instance $I$ is arc-consistent.
\item[(P2)] For every variable $x$, every set $B \subseteq A_x$, and every cycle $p$ from $x$ to $x$,
\[
B + p = B \;\; \implies \;\; B - p = B.
\]
\item[(P3)] For every variable $x$, every set $B \subseteq A_x$, and every pair of cycles $p,q$ from $x$ to $x$,
\[
B + p + q = B \;\; \implies \;\; B + p = B.
\]
\end{itemize}
\end{definition}

\noindent An alternative form of condition (P2) is given in \cite{sdp}.

\begin{proposition}[Barto, Kozik \cite{sdp}]\label{prop-p2*} If an instance satisfies condition (P1), then (P2) is equivalent to the following condition.
\begin{itemize}
\setlength{\itemindent}{0.5em}
\item[\emph{(P2*)}] For all variables $x$, sets $B \subseteq A_x$, and cycles $p$ from $x$ to $x$ with first step $s_1$ such that $B + p = B$, \[B + s_1 - s_1 = B;\]
\hspace{2pt} that is, $B$ is a union of linked components of $s_1$.
\end{itemize}
\end{proposition}

Conditions (P1) and (P2) are closely related to the basic linear programming relaxation of a CSP, while condition (P3) is closely related to the basic \emph{semidefinite} programming relaxation of a CSP (see \cite{sdp}).

\begin{theorem}\label{lp-p1-p2} If $I$ is an instance of a CSP such that the basic linear programming relaxation of $I$ has a solution assigning probability vectors $p_R$ to each constraint $R$ of $I$ and probability vectors $p_x$ to each variable $x$, then the instance $I'$ obtained by restricting each constraint relation of $I$ to the support of the corresponding probability distribution $p_R$, and similarly for the variable domains, satisfies conditions (P1) and (P2).
\end{theorem}
\begin{proof} Assume for simplicity that $I = I'$; that is, all of the probability vectors have full support. The compatibility of the probability vectors $p_R$ with the probability vectors on the variable domains ensures that $I$ is arc-consistent, so (P1) is satisfied. For (P2), it is easier to check condition (P2*) from Proposition \ref{prop-p2*}. For each set $B \subseteq A_x$, we attach a probability $P(B)$ given by
\[
P(B) \coloneqq \sum_{a \in B} p_{x,a}.
\]
Now consider any step $p_1$ from a variable $x$ to an adjacent variable $y$ within a constraint with corresponding relation $R$. Let $S \subseteq A_x \times A_y$ be the binary projection of the corresponding relation $R$ onto $x$ and $y$, and let $p_S$ be the corresponding marginal distribution of $p_R$. Then we have
\[
P(B + S) = \sum_{b \in B+S} p_{y,b} \ge \sum_{b \in B+S} \sum_{a \in B} p_{S,(a,b)} = \sum_{a \in B} p_{x,a} = P(B),
\]
with equality when $B + S - S = B$. Thus if $B + p = B$, then we have
\[
P(B) \le P(B+p_1) \le P(B+p) = P(B),
\]
so $P(B+p_1) = P(B)$, and thus we have $B + p_1 - p_1 = B$.
\end{proof}


In fact, Theorem \ref{lp-p1-p2} has a converse when we restrict our attention to a single cycle at a time. The proof is a straightforward generalization of the implication $(a) \implies (e)$ from Proposition \ref{bin-reversible}.

\begin{theorem} If $I$ is an instance of a CSP such that the associated hypergraph of variables and relations consists of a single cycle, then $I$ has properties (P1) and (P2) if and only if the basic linear relaxation of $I$ has a solution such that for each constraint $R$ of $I$, the support of the corresponding probability distribution $p_R$ is exactly equal to the relation corresponding to $R$.
\end{theorem}

The connections between the basic linear programming relaxation and conditions (P1) and (P2) make the following conjecture natural.

\begin{conjecture}\label{p1-p2} A CSP defined by relations $\Gamma$ is solved by its linear programming relaxation if and only if every instance $I$ of the CSP which satisfies conditions (P1) and (P2) has a solution.
\end{conjecture}

Conjecture \ref{reversible-conj} is a slight strengthening of a special case of Conjecture \ref{p1-p2}, where we restrict to the case of CSPs which have just a single variable and a single relation.


%% file: main.bbl
\begin{thebibliography}{1}

\bibitem{barto}
Libor Barto and Marcin Kozik.
\newblock Absorbing subalgebras, cyclic terms, and the constraint satisfaction
  problem.
\newblock {\em Log. Methods Comput. Sci.}, 8(1):1:07, 27, 2012.

\bibitem{sdp}
Libor Barto and Marcin Kozik.
\newblock Robust satisfiability of constraint satisfaction problems.
\newblock In {\em Proceedings of the Forty-fourth Annual ACM Symposium on
  Theory of Computing}, STOC '12, pages 931--940, New York, NY, USA, 2012. ACM.

\bibitem{weakprague}
Libor Barto and Marcin Kozik.
\newblock Constraint satisfaction problems solvable by local consistency
  methods.
\newblock {\em J. ACM}, 61(1):Art. 3, 19, 2014.

\bibitem{bulatov-jeavons-varieties}
Andrei Bulatov and Peter Jeavons.
\newblock Algebraic structures in combinatorial problems.
\newblock 2001.

\bibitem{bulatov}
Andrei Bulatov, Peter Jeavons, and Andrei Krokhin.
\newblock Classifying the complexity of constraints using finite algebras.
\newblock {\em SIAM J. Comput.}, 34(3):720--742, 2005.

\bibitem{dalmau-distributed}
Silvia Butti and Victor Dalmau.
\newblock The complexity of the distributed constraint satisfaction problem.
\newblock {\em arXiv preprint arXiv:2007.13594}, 2020.

\bibitem{carvalho}
Catarina Carvalho and Andrei Krokhin.
\newblock On algebras with many symmetric operations.
\newblock {\em Internat. J. Algebra Comput.}, 26(5):1019--1031, 2016.

\bibitem{subquotient}
Ralph Freese and Matthew~A Valeriote.
\newblock On the complexity of some maltsev conditions.
\newblock {\em International Journal of Algebra and Computation},
  19(01):41--77, 2009.

\bibitem{kun}
Gabor Kun, Ryan O'Donnell, Suguru Tamaki, Yuichi Yoshida, and Yuan Zhou.
\newblock Linear programming, width-1 {CSP}s, and robust satisfaction.
\newblock In {\em Proceedings of the 3rd {I}nnovations in {T}heoretical
  {C}omputer {S}cience {C}onference}, pages 484--495. ACM, New York, 2012.

\end{thebibliography}
